\documentclass[a4paper,10pt,reqno]{amsart}

\usepackage[utf8]{inputenc}
\RequirePackage[l2tabu, orthodox]{nag}
\usepackage[T1]{fontenc}
\usepackage{lmodern}
\usepackage{amsthm,amssymb,amsmath}
\usepackage{latexsym}
\usepackage{graphicx}
\usepackage{subfigure}
\usepackage{tikz-cd}
\usetikzlibrary{arrows}
\usetikzlibrary{calc}
\usepackage[linktocpage=true]{hyperref}
\usepackage{color}
\usepackage{here}
\usepackage{mathrsfs}
\usepackage{pgf,tikz}
\usetikzlibrary{decorations.pathreplacing, shapes.multipart, arrows, matrix, shapes}
\usetikzlibrary{patterns}
\usepackage{caption}
\usepackage[francais,english]{babel}
\usepackage[all] {xy}
\usepackage{enumitem}
\usepackage{array}
\usepackage{tabularx}

\makeatletter
\newcommand\footnoteref[1]{\protected@xdef\@thefnmark{\ref{#1}}\@footnotemark}
\makeatother

\hypersetup{ colorlinks   = true, 
urlcolor  = blue, 
linkcolor    = blue, 
citecolor   = blue 
}

\date{\today}

\def\beq{\begin{equation}}
\def\eeq{\end{equation}}

\def\PSL{\mathrm{PSL}_2(\mathbb{R})}

\def\PSLc{\mathrm{PSL}_2(\mathbb{C})}
\def\Hyp{\mathbb{H}^3}
\def\H2{\mathbb{H}^2}

\def\CP{\C\PP^1}
\def\RP{\R\PP^1}

\def\restriction#1#2{\mathchoice
              {\setbox1\hbox{${\displaystyle #1}_{\scriptstyle #2}$}
              \end{enumerate}
\restrictionaux{#1}{#2}}
              {\setbox1\hbox{${\textstyle #1}_{\scriptstyle #2}$}
              \restrictionaux{#1}{#2}}
              {\setbox1\hbox{${\scriptstyle #1}_{\scriptscriptstyle #2}$}
              \restrictionaux{#1}{#2}}
              {\setbox1\hbox{${\scriptscriptstyle #1}_{\scriptscriptstyle #2}$}
              \restrictionaux{#1}{#2}}}
\def\restrictionaux#1#2{{#1\,\smash{\vrule height .8\ht1 depth .85\dp1}}_{\,#2}}

\newcommand{\R}{{\mathbb R}}

\newcommand{\C}{{\mathbb C}}
\newcommand{\D}{{\mathbb D}}

\newcommand{\N}{{\mathbb N}}
\newcommand{\PP}{{\mathbb P}}

\newcommand{\loc}{{\operatorname{loc}}}

\newcommand{\cC}{{\mathcal C}}

\newcommand{\cF}{{\mathcal F}}

\newcommand{\cL}{{\mathcal L}}

\newcommand{\eps}{\varepsilon}

\def\dans{\mathop{\subset}}
\newcommand{\Leb}{\mathrm{Leb}}
\newcommand{\sect}{\mathrm{sect}}
\newcommand{\QF}{\mathrm{QF}}
\newcommand{\QC}{\mathrm{QC}}
\newcommand{\KD}{\mathrm{KD}}
\newcommand{\MKD}{\mathrm{MKD}}

\newcommand{\FKD}{\mathrm{FKD}}
\newcommand{\JC}{\mathrm{JC}}
\newcommand{\rD}{\mathrm{D}}

\newcommand{\Vol}{\mathrm{Vol}}

\newcommand{\bord}{\partial_{\infty}}

\newcommand{\Ent}{\mathrm{Ent}}

\newcommand{\Isom}{\mathrm{Isom}}
\newcommand{\Area}{\mathrm{Area}}

\newtheorem{theorem}{Theorem}[section]

\newtheorem{remark}[theorem]{Remark}

\newtheorem{lemma}[theorem]{Lemma}
\newtheorem{conjecture}[theorem]{Conjecture}
\newtheorem{defi}[theorem]{Definition}
\newtheorem{prop}[theorem]{Proposition}
\newtheorem{corollary}[theorem]{Corollary}

\newtheorem{question}[theorem]{Question}

\catcode`\@=11
\def\triplealign#1{\null\,\vcenter{\openup1\jot \m@th %
\ialign{\strut\hfil$\displaystyle{##}\quad$&$\displaystyle{{}##}$\hfil&$\displaystyle{{}##}$\hfil\crcr#1\crcr}}\,}
\def\multiline#1{\null\,\vcenter{\openup1\jot \m@th %
\ialign{\strut$\displaystyle{##}$\hfil&$\displaystyle{{}##}$\hfil\crcr#1\crcr}}\,}
\catcode`\@=12

\begin{document}
	
\title[]{Rigidity of the hyperbolic marked energy spectrum and entropy for $k$-surfaces}

\author[]{Sébastien Alvarez, Ben Lowe, Graham Andrew Smith}
\address{}
\email{}

\date{\today}

\maketitle

\selectlanguage{english}
\begin{abstract}
Labourie raised the question of determining the possible asymptotics for the growth rate of compact $k$-surfaces, counted according to energy, in negatively curved $3$-manifolds, indicating the possibility of a theory of thermodynamical formalism for this class of surfaces. Motivated by this question and by analogous results for the geodesic flow, we prove a number of results concerning the asymptotic behavior of high energy $k$-surfaces, especially in relation to the curvature of the ambient space.

First, we determine a rigid upper bound for the growth rate of quasi-Fuchsian $k$-surfaces, counted according to energy, and with asymptotically round limit set, subject to a lower bound on the sectional curvature of the ambient space.  We also study the marked energy spectrum for $k$-surfaces, proving a number of domination and rigidity theorems in this context. Finally, we show that the marked area and energy spectra for $k$-surfaces in $3$-dimensional manifolds of negative curvature are asymptotic if and only if the sectional curvature is constant.

\end{abstract}

\selectlanguage{french}
\begin{abstract}
Labourie a soulevé la question de déterminer les comportements asymptotiques possibles pour le taux de croissance des $k$-surfaces compactes, comptées selon leur énergie, dans les $3$-variétés de courbure négative, en suggérant la possibilité d'une théorie du formalisme thermodynamique pour cette classe de surfaces. Motivés par cette question et par des résultats analogues pour le flot géodésique, nous démontrons plusieurs résultats concernant le comportement asymptotique des $k$-surfaces de grande énergie, en particulier en lien avec la courbure de l'espace ambiant.

Premièrement, nous établissons une borne supérieure rigide pour le taux de croissance des $k$-surfaces quasi-Fuchsiennes, comptées selon leur énergie et ayant un ensemble limite asymptotiquement circulaire, sujet à une borne inférieure pour la courbure sectionnelle de l'espace ambiant. Nous étudions également le spectre marqué des énergies des $k$-surfaces, en prouvant plusieurs théorèmes de domination et de rigidité dans ce contexte. Enfin, nous montrons que les spectres marqués des aires et des énergies des $k$-surfaces dans les $3$-variétés de courbure négative sont asymptotiques si et seulement si la courbure sectionnelle est constante.
\end{abstract}
\selectlanguage{english}

\tableofcontents

\section{Introduction}

\subsection{Context}
This paper continues our study of the dynamical properties of the space of immersed surfaces of constant, positive extrinsic curvature - henceforth called $k$-surfaces - in closed, negatively-curved $3$-manifolds. The study of this space was initiated by Labourie in a remarkable series of papers \cite{LabourieGAFA,LabourieInvent,LabourieAnnals} (see also \cite{Labourie_phase_space}). There, building on Gromov's theory of foliated Plateau problems elaborated in \cite{Gromov_FolPlateau1,Gromov_FolPlateau2}, he showed that, for ambient manifolds of sectional curvature bounded above by $-1$, and for $0<k<1$, the space of marked $k$-surfaces exhibits many of the hyperbolic properties of the geodesic flow. More precisely:
\begin{enumerate}
\item it has a natural compactification which is laminated by Riemann surfaces;
\item it is independent of the ambient metric up to leaf-preserving homeomorphism, in analogy to Gromov's geodesic rigidity theorem \cite{Gromov3} (a local version was proven by Labourie in the geometrically finite case \cite{LabourieInvent}, and a global version was recently proven in the general case by the third author in \cite{Smith_asymp});
\item the generic leaf is dense, and the union of closed leaves is also dense;
\item the lamination admits an infinite family of mutually singular, ergodic, totally invariant measures of full support  obtained by a coding procedure.
\end{enumerate}

In this vein, Labourie suggested a possible analogy between the thermodynamic properties of the geodesic flow in negative curvature and the thermodynamic properties of the space of $k$-surfaces \cite{LabourieInvent,Labourie_phase_space}. More precisely, in Question $(ii)$ of Section $1.2$ of \cite{LabourieInvent}, he posed the question of determining the possible asymptotic growth rates of the number of closed $k$-surfaces in a  compact, negatively-curved $3$-manifold, counted according to their energies \footnote{Note that what we refer to in the present paper as ``energy'' is referred to by Labourie in \cite{LabourieInvent} as ``area''.}.  Little progress has been made towards a satisfactory answer to his question up to this point.    However, following the groundbreaking work \cite{CMN} of Calegari--Marques--Neves (see also \cite{LabourieBourbaki}), and continuing in the spirit of \cite{BCG}, we obtained in \cite{ALS} a domination and rigidity result for the asymptotic growth rate according to their {\sl areas} of the number of closed {\sl quasi-Fuchsian} $k$-surfaces in any such manifold. In this paper, we build on these ideas to prove domination and rigidity results for the asymptotic growth rates for this class of surfaces, counted according to their \textit{energies}. We view this as partially addressing Labourie's question.

\subsection{Marked area and energy spectra for $k$-surfaces}

Before stating our main results, we first require a few definitions. Let $(X,h_0):=\Bbb{H}^3/\Pi$ be a compact, $3$-dimensional hyperbolic manifold, where $\Pi$ is a discrete, torsion-free subgroup of $\Isom^+(\Hyp)=\PSLc$. Given $C\geq 1$, we say that a compact surface subgroup $\Gamma$ of $\Pi$ is $C$-\emph{quasi-Fuchsian} whenever its limit set is a $C$-quasicircle. The existence of large families of quasi-Fuchsian subgroups was proven by Kahn--Markovi\'c in \cite{KM1,KM2} (see also \cite{HamenstadtKM,Kahn_Labourie_Mozes}). Let $\QF(C)$ denote the set of conjugacy classes of oriented $C$-quasi-Fuchsian subgroups of $\Pi$, and denote
\begin{equation*}
\QF:=\bigcup_{C\geq 1}\QF(C)\ .
\end{equation*}
Let $h$ be another riemannian metric on $X$ of sectional curvature less than $-k$, for some $k>0$. It follows from the main results of \cite{LabourieInvent} (see also \cite{Smith_asymp}) that, for every oriented conjugacy class $[\Gamma]\in\QF$, there exists a unique closed $k$-surface in $X$ representing $[\Gamma]$, which we henceforth denote by $S_{k,h}([\Gamma])$.

We find that our results are best contextualized by the following reformulation in terms of the marked area spectrum of the main result of \cite{ALS}. We define the \emph{marked (quasi-Fuchsian) $k$-surface area spectrum} to be the map $\text{MAS}_{k,h}:\QF\rightarrow]0,\infty[$ given by
\begin{equation*}
\text{MAS}_{k,h}([\Gamma]) := \Area_h(S_{k,h}([\Gamma]))\ .
\end{equation*}
The main result of \cite{ALS} can now be expressed as follows.

\begin{theorem}[Domination and rigidity of the hyperbolic marked area spectra for $k$-surfaces]\label{th.rigidity_area}
Let $(X,h_0):=\Bbb{H}^3/\Pi$ be a closed hyperbolic $3$-manifold, and let $h$ be another riemannian metric on $X$ of sectional curvature bounded above by $-1$. For all $0<k<1$,
\begin{equation}\label{eqn.rigidity_area}
\mathrm{MAS}_{k,h}\leq\mathrm{MAS}_{k,h_0}\ ,
\end{equation}
with equality holding if and only if $h$ is hyperbolic.
\end{theorem}

\begin{remark} In particular, by Mostow's rigidity theorem \cite{Mostow}, equality holds if and only if $h$ is isometric to $h_0$.
\end{remark}

For any compact, locally strictly convex immersed surface $S\dans (X,h)$, and for all $p\geq 0$, we now define the \emph{$p$-energy} of $S$ by
\begin{equation}\label{eq.energy}
W^p_h(S):=\int_SH^pd\Area_h\ ,
\end{equation}
where $H$ denotes the mean curvature of $S$ (which, by local strict convexity, we may take to be positive). Note that the $0$-energy of $S$ is simply its area, whilst its $1$-energy is the area with respect to the Sasaki metric of its \emph{Gauss lift} to the unit sphere bundle $S_hX$ (see Section \ref{ss.as_Plateau_definitions}). We define the \emph{marked (quasi-Fuchsian) $k$-surface $p$-energy spectrum} to be the map $\text{MES}_h^p:\QF\rightarrow]0,\infty[$ given by
\begin{equation}\label{eq.MarkedEnergySpectrum}
\text{MES}_h^p([\Gamma]) := W^p_h(S_{k,h}([\Gamma]))\ .
\end{equation}
In particular, the marked $0$-energy spectrum is just the marked area spectrum defined above. Our first main theorem is the following rigidity result for marked energy spectra.

\begin{theorem}[Rigidity of the hyperbolic marked $p$-energy spectrum]\label{thm.Rigidity_hyp_energy_spectrum}
Let $(X,h_0):=\Bbb{H}^3/\Pi$ be a closed hyperbolic $3$-manifold, and let $h$ be another riemannian metric on $X$ of sectional curvature pinched between $-1$ and $-a$, for some $0<a<1$. For all $0<k<a$, the equality of marked energy spectra
\begin{equation}\label{eqn.rigidity_energy}
\mathrm{MES}^p_{k,h_0}=\mathrm{MES}^p_{k,h}\ ,
\end{equation}
holds if and only if $h$ is hyperbolic.
\end{theorem}

\begin{remark}
Note that this rigidity result requires that the sectional curvature of $h$ be bounded from \emph{below} rather than from above, which has been the case for all previous results of this kind (c.f. \cite{ALS,CMN}).
\end{remark}

The proof makes use of the following straightforward but interesting fact: \emph{in constant sectional curvature, nearly Fuchsian $k$-surfaces are almost totally umbilical}. It turns out that this property characterizes metrics of constant curvature. Indeed, we define the \emph{normalized $p$-energy} of a closed $k$-surface $S\dans X$ by
\begin{equation*}
\overline{W}^{p}_h(S) := k^{-\frac{p}{2}}\int_S H^p d\Area_h\ .
\end{equation*}
For $p\neq q\geq 0$ we say that the normalized marked $k$-surface $p$- and $q$-energy spectra are \emph{asymptotic} whenever, for every sequence $(\eps_n)_{n\in\N}$ of positive numbers converging to $0$, and every sequence $(S_n)_{n\in\N}$ of closed $(1+\eps_n)$-quasi-Fuchsian $k$-surfaces,
\begin{equation*}
\lim_{n\to\infty}\frac{\overline{W}^{p}_h(S_n)}{\overline{W}^{q}_h(S_n)}=1\ .
\end{equation*}

\begin{theorem}[Asymptotic energy spectra]\label{thm.asymptotic_spectra}
Let $(X,h)$ be a closed $3$-manifold of sectional curvature bounded above by $-a$ for some $a>0$. For all $0<k<a$, the following properties are equivalent.
\begin{enumerate}
\item $h$ has constant sectional curvature;
\item for all $p\neq q\geq 0$ the normalized marked $k$-surface $p-$ and $q$-energy spectra of $k$-surfaces are asymptotic; and
\item for some $p\neq q\geq 0$ the normalized marked $k$-surface $p-$ and $q$-energy spectra of $k$-surfaces are asymptotic.
\end{enumerate}
\noindent In particular, $h$ has constant sectional curvature if and only if its normalized marked $k$-surface energy and area spectra are asymptotic.
\end{theorem}

We point out that besides the curvature upper bound of $-a$, which is necessary to ensure that $k$-surfaces exist in the first place, the previous theorem, unlike the other theorems stated before, has no assumptions on sectional curvature.

\subsection{Area and energy entropies for $k$-surfaces} The method of proofs of the results above are similar to those of \cite{ALS}. We use the solution of foliated Plateau problems, and the equidistribution properties of closed quasi-Fuchsian $k$-surfaces. This is the method outlined by Labourie in his influential Bourbaki seminar \cite{LabourieBourbaki} in the case of minimal surfaces.

These results can also be phrased in terms of counting. We will define in \S \ref{sss.Modified_entropy} an \emph{entropy} similar to those appearing in \cite{ALS,CMN}. It is actually analogous to the \emph{modified entropy} defined by Marques-Neves in \cite{MN_currents} and is a number, denoted by $\Ent^p_{k,\hat m}(X,h)$, that counts, according to their $p$-energies, the closed quasi-Fuchsian $k$-surfaces whose limit quasicircles becomes more and more round \emph{and equidistribute} to $\hat m$, the unique ergodic fully supported \emph{conformal current in the space $\cC^+$ of round circles}. We refer to \S \ref{sss.conf_currents_PSL_inv_laminar_measures} for the definition of conformal currents as defined by Labourie in \cite{LabourieBourbaki} (see also \cite{ALS,MN_currents}) and to \S \ref{sss.Modified_entropy} for the definition of entropy. The analogue of Theorem \ref{thm.Rigidity_hyp_energy_spectrum} for entropy is a rigid inequality.

\begin{theorem}[Rigid inequality]\label{th.rigidity_entropy}
Let $(X,h_0)$ be a closed hyperbolic $3$-manifold and $h$ be a Riemannian metric on $X$ with sectional curvature  $-1\leq\sect_h\leq -a$ for some $a>0$. Let $k\in (0,a)$ and $p\geq 0$. Then  $\Ent^p_{k,\hat m}(X,h)\leq\Ent^p_{k,\hat m}(X,h_0)$ and equality holds if and only if $h$ and $h_0$ are isometric.
\end{theorem}

\begin{remark}
Unlike the case of the marked area spectra, we don't get a domination of marked energy spectra, since we control only the energy of almost Fuchsian $k$-surfaces in $\Hyp$, which are nearly totally umbilical. The domination is thus asymptotic and appears at the level of entropies. So the situation is similar to that of minimal surface areas in \cite{CMN}, contrasting with our result in \cite{ALS}.
\end{remark}

The analogue of Theorem \ref{thm.Rigidity_hyp_energy_spectrum} is the following theorem proving entropy rigidity (we refer to \S \ref{sss.equality_modified_entropies} for the definition of $\Ent^\mathrm{Area}_{k,\hat m}(X,h)$, the modified area entropy for $k$-surfaces).

\begin{theorem}[Entropy rigidity]\label{equality_energy_area}
Let $(X,h)$ be a closed $3$-manifold with sectional curvature  $\sect_h\leq -a$ for some $0<a$. Let $k\in (0,a)$. Then the following properties are equivalent.
\begin{enumerate}
\item The sectional curvature of $h$ is constant.
\item For some $p\neq q\geq 0$ we have
$$k^{-p/2} \Ent^p_{k,\hat m}(X,h)=k^{-q/2}\Ent^q_{k,\hat m}(X,h).$$
\item The modified area entropy and modified $p$-energy entropy coincide for some $p\neq 0$
$$\Ent^\mathrm{Area}_{k,\hat m}(X,h)=k^{-p/2}\Ent^p_{k,\hat m}(X,h).$$
\end{enumerate}
\end{theorem}

These results give the first counting result of closed $k$-surfaces according to their energy, which we view as partially addressing Labourie's original question \cite[Question (ii) p.243]{LabourieInvent}.

We note that the previous two results continue to hold with the modified entropies replaced by the standard entropies that count $k$-surfaces whose limit sets become more and more round but without an equidistribution condition, under a topological condition on $X$: that $X$ contain no closed totally geodesic surfaces in its hyperbolic metric (see Question \ref{question:intro}, Remark \ref{remarkattheend}). Examples of closed hyperbolic 3-manifolds without closed totally geodesic surfaces are given in \cite{MR03}. Let us note that many hyperbolic knot complements have no totally geodesic surfaces: see \cite{Basilio_Lee_Malionek,Reid} and references therein.

\subsection{Background and Future Directions}

\subsubsection{Marked length spectrum and domination}

Our work is motivated in part by the theory developed around \textit{marked length spectrum rigidity.} Recall that, for any closed $3$-dimensional Riemannian manifold $(X,h)$ with negative sectional curvature, we define its marked length spectrum to be the function $\text{MLS}_h$ which associates to every conjugacy class of $\Pi=\pi_1(X)$ the length of its geodesic representative. The following conjecture (see \cite[Problem 3.1]{BurnsKatok}) remains open.
\begin{conjecture}[Rigidity of the marked length spectrum]
Let $h_1$ and $h_2$ be two Riemannian metrics with negative sectional curvature on a closed manifold $X$. Then $\mathrm{MLS}_{h_1}=\mathrm{MLS}_{h_2}$ if and only if $h_1$ and $h_2$ are isometric, that is, if and only if there exists a diffeomorphism $\phi:X\to X$ such that $\phi^\ast h_2=h_1$.
\end{conjecture}

This conjecture has been proved for surfaces by Otal \cite{Otal} (see also \cite{Croke} for results in non-positive curvature). In dimension $3$ it has been established in two particular cases
\begin{enumerate}
\item if one of the metrics has constant sectional curvature: see \cite{besson1996minimal,HamenstadtMLS};
\item if the two metrics are close enough: see \cite{GuillarmouLefeuvre}.
\end{enumerate}

Recently, some variations on the marked length spectrum rigidity problem have been considered. For example there is the problem of understanding the consequences of the property of \emph{domination} of marked length spectra $\text{MLS}_h\geq\text{MLS}_{h'}$. Conjecturally this should imply that for two metrics $h$ and $h'$ of negative sectional curvatures, with the sectional curvatures of $h$ bounded below by those of $h'$, that $\Vol(X,h)\geq\Vol(X,h')$ (see \cite{Croke_Dairbekov,Croke_Dairbekov_Sharafutdinov} as well as \cite{Gogolev_Reber} for related rigidity results). This domination property was studied in the context of representation of surface groups in \cite{Deroin_Tholozan,Gueritaud_Kassel_Wolff} and later generalized to other contexts, such as harmonic maps, higher Teichm\"uller theory and partial hyperbolicity \cite{Alvarez_Yang,Barman_Gupta,Dai_Li,Li,Sagman}.

A direct analogue of marked length spectrum rigidity for $k$-surfaces is marked energy spectrum rigidity for $\pi_1$-injective $k$-surfaces, that is, where we consider the function which assigns to each conjugacy class of surface subgroups of $\Pi$ the energy of the corresponding $k$-surface in the given negatively curved metric on $X$. Establishing marked energy spectrum rigidity even infinitesimally is an interesting and hard problem. If instead of the energy, one considers the map which assigns to each conjugacy class of surface subgroups of $\Pi$ the quotient of the $p$ and $q$-energies of the corresponding $k$-surface, for some $p \neq q$, then Theorem \ref{thm.asymptotic_spectra} implies the corresponding marked spectral rigidity statement for metrics of negative sectional curvature, one of which has constant curvature.

\subsubsection{Thermodynamical formalism for $k$-surfaces}

The marked length spectrum rigidity and domination results for the geodesic flow fit into the theory of Hölder cocycles developed by Ledrappier and Hamenst\"adt \cite{HamenstadtMLS,Ledrappier_bord}. In this theory, one associates two objects to every real-valued Hölder function $F$ over the sphere bundle $S_hX$, henceforth called a \emph{potential}. The first is a function, called the \emph{spectrum}, mapping every conjugacy class of $\Pi$ to the integral of $F$ over its unique geodesic representative, and the second is a probability measure, defined by the \emph{variational principle}, and invariant under the geodesic flow, which we refer to as an \emph{equilibrium state}. By Liv\v{s}ic's theorem (see \cite{Guillemin_Kazhdan,Livsic}), up to a suitable normalization, two potentials have the same periods if and only if they are cohomologous, which in turn holds if and only if they have the same equilibrium states.

It is classical to consider three potentials in $S_hX$: the constant potential (associated to the maximal entropy measure), the geometric potential (associated to the Liouville measure) and the harmonic potential (associated to the harmonic measure). In dimension $2$, it is known (see \cite{Katok_entropy_closed,Katok4,Ledrappier_har_BM}) that if any two of these three potentials have the same periods then the ambient metric has constant curvature (the only known results in higher dimension hold for perturbations of hyperbolic metrics and are due to Flaminio \cite{Flaminio} and Humbert \cite{humbert2024katok}).

It is natural to ask whether there is a similar theory of thermodynamical formalism applying to $k$-surfaces containing Theorems \ref{thm.asymptotic_spectra} and \ref{equality_energy_area}. Indeed, a similar question is raised by Labourie in \cite[Question 8]{Labourie_phase_space}. A concrete starting point for such a theory would be the analogue of Liv\v{s}ic's theorem in the context of $k$-surfaces.

\begin{question}
Let $X$ be a closed negatively curved manifold, let $F$ be a Hölder function defined on the space of marked frame bundles $FS$ of quasi-Fuchsian $k$-surfaces $S$ of $X$. Assume that for every quasi-Fuchsian $k$-surface $S$ the integral of $F$ over $FS$ vanishes. What can be concluded about $F$?
\end{question}

This question seems to belong to a cohomological theory of $\PSL$-actions, though it is not immediately clear what should be considered as a coboundary in this context. One of the main difficulties is the absence in the case of $k$-surfaces of a theory of hyperbolic dynamics, which has proven invaluable in the study of the geodesic flow. A key challenge lies in identifying what may serve as a substitute.

The main respect in which our work is less general than the setting proposed by Labourie's question is that we focus on quasi-Fuchsian surfaces which, in particular, are $\pi_1$-injective. It would be interesting to determine the extent to which the results of this paper hold for entropy functionals defined using larger families of compact $k$-surfaces.

\subsubsection{Closed totally umbilical surfaces}

Finally we state a question whose affirmative answer would allow us to extend the results of Section \ref{sss.Modified_entropy} to more general entropy functionals.

\begin{question} \label{question:intro}
Let $X$ be a closed hyperbolizable riemannian 3-manifold with negative sectional curvature greater than or equal to $-1$. Suppose that $X$ contains a closed totally umbilical surface $S$ with mean curvature $k^{1/2}$ and with constant curvature $-1+k$ in its intrinsic metric (that is, the sectional curvature of $X$ through every tangent plane of $S$ is equal to $-1$). Does it follow that $X$ is hyperbolic?
\end{question}

\subsection*{Acknowledgements}

{\footnotesize S.A. was supported by CSIC, via the Grupo I+D 149 ``Geometr\'ia y acciones de Grupos''. He was also supported by the IRL-2030 IFUMI, Laboratorio del Plata. B.L. was supported by the NSF under grant DMS-2202830. We wish to thank the anonymous referee for his/her valuable comments}.

\section{Foliated Plateau problems and invariant measures}

\subsection{Foliated Plateau problems and actions of $\PSL$}\label{ss.as_Plateau}

\subsubsection{Setting and definitions}\label{ss.as_Plateau_definitions} Throughout the sequel, $X$ will denote a closed, connected, hyperbolic $3$-manifold, $h_0$ will denote its hyperbolic metric, $\tilde{X}$ will denote its universal cover, and $\Pi$ will denote its fundamental group. We identify $\tilde{X}$ with hyperbolic $3$-space $\Hyp$, and we identify $\Pi$ with a discrete subgroup of $\Isom^+(\Hyp)=\PSLc$. $\tilde{X}=\Hyp$ is naturally compactified by its ideal boundary $\bord\Hyp$, which in turn identifies with the Riemann sphere $\CP$. We thus denote this compactification by $\tilde{X}\cup\CP$.

By Mostow's rigidity theorem \cite{Mostow}, $h_0$ is unique up to diffeomorphism. Our aim is to characterize this metric within the space of riemannian metrics on $X$ of sectional curvature pinched between $-b$ and $-a$, for $0<a<b$. Let $h$ be such a riemannian metric, and let $h$ also denote its lift to the universal cover $\tilde{X}$.

We will be interested in the following two bundles over $\tilde{X}$.
\begin{itemize}
\item \emph{The unit sphere bundle} $S_h\tilde{X}\dans T\tilde{X}$, whose fibre at the point $x$ is
\begin{equation*}
S_h\tilde{X}_x:=\{v\ |\ \|v\|_h=1\}\ .
\end{equation*}
\item \emph{The frame bundle} $F_h\tilde{X}\dans T\tilde{X}\oplus T\tilde{X}\oplus T\tilde{X}$, whose fibre at the point $x$ is
\begin{equation*}
F_h\tilde{X}_x:=\{\xi:=(v,e_1,e_2)\ |\ (v,e_1,e_2)\ \text{oriented and $h$-orthonormal}\}\ .
\end{equation*}
\end{itemize}
There is natural projection $p:F_h\tilde{X}\rightarrow S_h\tilde{X}$ given by
\begin{equation*}
p(v,e_1,e_2):=v\ ,
\end{equation*}
making $F_h\tilde{X}$ into a bundle over $S_h\tilde{X}$ with fibre $\Bbb{S}^1$. Furthermore, since $\Pi$ preserves $h$, it acts on both $S_h\tilde{X}$ and $F_h\tilde{X}$ by differentials of isometries, and the projection $p$ is equivariant with respect to this action. We denote by $S_hX$ and $F_hX$ the respective quotients of these bundles, and we observe that they are respectively the unit sphere bundle and the unit frame bundle of $(X,h)$.

Given an oriented, embedded disk $D\dans\tilde{X}$, with unit normal vector field $n:D\rightarrow S_h\tilde{X}$, we define its \emph{Gauss lift} $\hat{D}\dans S_h\tilde{X}$ by
\begin{equation*}
\hat{D} := \{n(p)\ |\ p\in D\}\ .
\end{equation*}
Given a point $p\in D$, we define a \emph{frame} of $D$ at $p$ to be a triple $\xi:=(n(p),e_1,e_2)$, where $(e_1,e_2)$ is an oriented, orthonormal pair of tangent vectors of $D$ at $p$. We define the \emph{frame bundle} $FD\dans F_h\tilde{X}$ by
\begin{equation*}
FD := \{\xi(p)\ |\ \xi\ \text{a frame of $D$ at some point $p$}\}\ .
\end{equation*}
The Gauss lift is naturally diffeomorphic to $D$, whilst the frame bundle is naturally diffeomorphic to the unit circle bundle $SD$.

We define a \emph{marked disk} in $\tilde{X}$ to be a pair $(D,p)$ where $D\dans\tilde{X}$ is an oriented disk, and $p\in D$. It will be convenient to identify each marked disk $(D,p)$ with its marked \emph{Gauss lift} $(\hat{D},n(p))$. Likewise, we define a \emph{marked frame bundle} of a disk $D$ to be a pair $(FD,\xi)$ where $FD$ is its frame bundle, and $\xi:=(n(p),e_1,e_2)$ is a frame of $FD$ based at some point $p\in D$.

\subsubsection{The space of $k$-disks}\label{sss.space_k_discs} Fix $0<k<a$. We define a \emph{$k$-disk} in $(\tilde{X},h)$ to be a smoothly embedded, oriented disk $D\dans \tilde{X}$, whose extrinsic curvature with respect to $h$ is constant and equal to $k$, and whose closure in $\tilde{X}\cup\CP$ is a closed topological disk. Note that since $D$ has positive extrinsic curvature, it is locally strictly convex, and we choose its orientation such that its unit normal vector field $n$ points outwards from the convex set that it bounds. In particular, with this convention, its mean curvature $H$ is everywhere positive. Let $\KD^+_{h,k}$ denote the space of all \emph{oriented} $k$-disks in $(\tilde{X},h)$, endowed with the $C^\infty_{loc}$-topology.

Recall that the \emph{intrinsic} and \emph{extrinsic curvatures} of any immersed surface $D$ in $(\tilde{X},h)$ are related by \emph{Gauss' equation}
\begin{equation}
\kappa_{int}=\kappa_{ext}+\sect_h|_{TD}\ .
\end{equation}
In particular, the intrinsic curvature of any $k$-disk satisfies
\begin{equation}\label{eq_intr_curv}
k-b\leq\kappa_{int}\leq k-a<0\ ,
\end{equation}
so that every $k$-disk is conformally of hyperbolic type.

\subsubsection{Asymptotic Plateau problem}\label{sss.asymptotic_Plateau_problem} We denote by $\JC^+$ the space of all oriented Jordan curves in $\CP$ endowed with the Hausdorff topology. Given $c\in\JC^+$, we say that $c$ \emph{spans} a $k$-disk $D$ whenever $\bord D=c$, with $\bord D$ furnished with the orientation that it inherits from $D$. We will require the following result from \cite[Theorem \& Definition 2.1.1]{ALS}, which is a corollary of the main results of \cite{LabourieInvent} (in the geometrically finite case) and \cite{Smith_asymp} (in the general case).

\begin{theorem}[Asymptotic Plateau problem]\label{th.As_plateau_problem}
For every $0<k<a$, and for every $c\in\JC^+$ there exists a unique $k$-disk $D:=\rD_{k,h}(c)$ which spans $c$. Furthermore, the function $c\in\JC^+\mapsto\rD_{k,h}(c)\in\KD^+_{k,h}$ is continuous.
\end{theorem}

\subsubsection{Quasicircles}\label{sss.PSL_quasicircles} Let $C\geq 1$. Recall that a Jordan curve $c$ in $\CP$ is a $C$-\emph{quasicircle} whenever it is the image of the real projective line $\RP$ under the action of some $C$-quasiconformal map.\footnote{We refer the reader to Ahlfors' book \cite{Ahlfors} for the basic properties of quasiconformal maps.} Let $\QC^+(C)$ denote the space of oriented $C$-quasicircles. In what follows, we will use the following properties of quasicircles and quasiconformal maps.
\begin{enumerate}
\item The composition of a $C$-quasiconformal map with a M\"obius transformation, on the left or the right, is also a $C$-quasiconformal map, and the image of any $C$-quasicircle under a Möbius transformation is also a $C$-quasicircle;
\item $1$-quasiconformal maps are Möbius transformations and $1$-quasicircles are round circles of $\CP$;
\item Any closed subset of $\QC^+(C)$ is compact if and only if it does not accumulate on a singleton.
\end{enumerate}
In particular, by $(1)$, for all $C$, the group $\Pi$ sends $\QC^+(C)$ to itself.

\subsubsection{Space of $k$-disks spanned by quasicircles}
For $C\geq 1$ and $0<k<a$, we will use the following two spaces.
\begin{itemize}
\item The space $\FKD_{k,h}(C)$, formed of all marked frame bundles $(FD,\xi)$, where $D$ is an oriented $k$-disk spanned by a $C$-quasicircle and $\xi\in FD$. We furnish this space with the $C^\infty_\loc$-topology.
\item The space $\MKD_{k,h}(C)$, formed of all Gauss lifts $(\hat D,n)$, where $D$ is an oriented $k$-disk spanned by a $C$-quasicircle and $n\in\hat D$. We likewise furnish this space with the $C^\infty_\loc$-topology.
\end{itemize}
$\Pi$ acts on these spaces cocompactly on the left (see \cite{ALS,LabourieGAFA}), and the canonical projection $p:\FKD_{k,h}(C)\to\MKD_{k,h}(C)$ is continuous, surjective, and $\Pi$-equivariant.

We will also use the following boundary maps
\begin{align*}
&\bord:\FKD_{k,h}(C)\to\QC^+(C);(FD,\xi)\mapsto\bord D\ \text{and}\\
&\bord:\MKD_{k,h}(C)\to\QC^+(C);(\hat{D},n)\mapsto\bord D\ .
\end{align*}
In \cite{ALS}, we show that these maps are continuous.

\subsubsection{Laminations, simultaneous uniformization, and the action of $\PSL$}
\label{sssec:LaminationsSimultaneousUniformization}
Since $k$-disks have intrinsic curvature bounded above by $k-a<0$ (see \S \ref{sss.space_k_discs}), they are uniformized by the hyperbolic plane $\H2$. The space $\text{MKD}_{k,h}(C)$ therefore carries a natural lamination by hyperbolic surfaces, whose leaves are the preimages of $\bord$, and which we denote by $\cL_{k,h}$. Since $\Pi$ sends leaves to leaves, the quotient $\text{MKD}_{k,h}(C)/\Pi$ also carries a hyperbolic surface lamination which, furthermore, is compact (see \cite{LabourieGAFA}). We denote this lamination also by $\cL_{k,h}$ when no ambiguity arises. In a similar manner, the space $\FKD_{k,h}(C)$ carries a natural lamination by $3$-manifolds, whose leaves are likewise the preimages of $\bord$, and which we denote by $F\cL_{k,h}$. We also denote the natural lamination of the quotient $\FKD_{k,c}(C)/\Pi$ by $F\cL_{k,h}$ when no ambiguity arises. Finally, note that, in both cases, the leaves of $F\cL_{k,h}$ are the unit circle bundles of the leaves of $\cL_{k,h}$.

There is a natural right action of $\PSL$ on $\FKD_{k,h}(C)$ which we now describe. We first show how every frame $\xi$ in $\FKD_{k,h}(C)$ naturally identifies with a uniformizing map of the leaf of $\MKD_{k,h}(C)$ which passes through its base-point. Indeed, identifying $\H2$ with the upper half-space in $\Bbb{C}$ in the standard manner, we define the oriented frame $(e_1^0,e_2^0)$ of this space at the point $i:=\sqrt{-1}$ by
\begin{equation*}
e_1^0 := 1\qquad\text{and}\qquad e_2^0 := i\ .
\end{equation*}
Given a $k$-disk $D$ and a frame $\xi:=(n(p),e_1,e_2)\in FD$, we now define $u_\xi:\H2\rightarrow D$ to be the unique conformal map such that
\begin{equation*}
u_\xi(i)=p\qquad\text{and}\qquad Du_\xi(e_1^0)=\lambda e_1\,\,\text{for some }\lambda>0\ .
\end{equation*}
The right action of $\PSL$ is now defined as follows. For all $g\in\PSL$, there exists a unique function $\alpha_k(g):\FKD_{k,h}(C)\rightarrow\FKD_{k,h}(C)$ such that, for every $k$-disk $D$, and for every frame $\xi:=(n(p),e_1,e_2)\in FD$,
\begin{equation*}
u_\xi\circ g = u_{\alpha_k(g)(\xi)}\ .
\end{equation*}
More explicitly,
\begin{equation*}
\alpha_k(g)(\xi) = (n(q),e_1',e_2')\ ,
\end{equation*}
where
\begin{equation*}
q := u_\xi(g(0))\qquad\text{and}\qquad e_1'=\lambda Du_\xi\cdot Dg\cdot e_1^0\,\,\text{for some }\lambda>0.
\end{equation*}
It is straightforward to show that this action of $\PSL$ on $\FKD_{k,h}(C)$ is free, transitive on each leaf, and commutes with the left action of $\Pi$. Furthermore, by Candel's simultaneous uniformization theorem  (in particular, the version given in \cite[Th\'eor\`eme 2.7]{AlvarezSmith}), the correspondence $\xi\mapsto u_\xi$ varies smoothly over every fibre, and continuously in the $C^\infty_{\mathrm{loc}}$ sense transversally to every fibre. In particular,  $\alpha_k(g)$ is a homeomorphism for all $g$, and the action is smooth over every fibre, and varies continuously in the $C^\infty_{\mathrm{loc}}$ sense transversally to every fibre. Thus this action is continuous.

Using this $\PSL$-action, we proved in \cite[Theorems 3.1.2. and 3.2.2]{ALS} the following result.

\begin{theorem}[Fibered Plateau problem]\label{t.fibered_PP}
For all $C\geq 1$, the map
\begin{equation*}
\bord :\FKD_{k,h}(C)\to\mathrm{QC}^+(C)\ ,
\end{equation*}
is a continuous principal $\PSL$-bundle over $\QC^+(C)$, and the map
\begin{equation*}
\bord :\mathrm{MKD}_{k,h}(C)\to\mathrm{QC}^+(C)\ ,
\end{equation*}
is a continuous disk-bundle, which is an associated bundle of $\FKD_{k,h}(C)$.
\end{theorem}

\subsubsection{Foliated Plateau problem}\label{sss.FPP} Consider now the case where $C=1$. Let $\cC^+=\text{QC}^+(1)$ denote the set of oriented round circles in $\CP$. In this case, the spaces $\FKD_{k,h}(1)$ and $\MKD_{k,h}(1)$ identify respectively with the frame bundle $F_h\tilde{X}$ and the sphere bundle $S_h\tilde{X}$, as the following theorem, proven in \cite[Theorem 2.1.3]{ALS}, shows.

\begin{theorem}[Foliated Plateau problem]\label{th_foliated_Plateau} The Gauss lifts of oriented $k$-disks spanned by round circles of $\CP$ form a continuous foliation $\cF_{k,h}$ of $S_h\tilde{X}$, and their frames form a continuous foliation $F\cF_{k,h}$ of $F_h\tilde{X}$. Furthermore, these foliations are $\Pi$-invariant.
\end{theorem}

We call the above foliations the $k$-\emph{surface foliation} of $S_h\tilde{X}$ and the $k$-\emph{frame foliation} of $F_h\tilde{X}$ respectively. By $\Pi$-invariance, they descend respectively to foliations of $S_h X$ and $F_h X$ whose respective leaves have dimension $2$ and $3$. The above theorem has the following consequences.

\begin{enumerate}
\item The laminated space $(\MKD_{k,h}(1),\cL_{k,h})$ identifies with the $k$-surface foliation of the sphere bundle $(S_h\tilde{X},\cF_{k,h})$. This identification is equivariant under the natural left action of $\Pi$.
\item The laminated space $(\FKD_{k,h}(1),F\cL_{k,h})$ identifies with the foliation of the frame bundle $(F_h\tilde{X},F\cF_{k,h})$. This identification is also equivariant by the natural left actions of $\Pi$.
\item $F_h\tilde{X}$ inherits a free right-action of $\PSL$, which we also denote by $\alpha_k$. The orbit foliation of this action identifies with the $k$-frame foliation $F\cF_{k,h}$.
\end{enumerate}

We emphasize the strong analogies between the preceding constructions and those arising from geodesics. Indeed, we view the action of $\PSL$ on $F_h\tilde X$ as a higher-dimensional analogue of the geodesic flow. We likewise view the $k$-surface foliation of $S_h\tilde{X}$ as a higher-dimensional analogue of the geodesic foliation of this manifold. This analogy is further justified, for example, by the fact that the $k$-surface foliation is independent of the metric in the following sense. For $k<a$, if $h$ and $h'$ are two negatively curved metrics on $X$ with sectional curvatures bounded above by $-a<0$, then there exists a canonical homeomorphism of $S_h\tilde{X}$ conjugating the $k$-surface foliations (see \cite[Theorem 2.1.5.]{ALS}). We view this as the analogue in the present context of the \emph{geodesic rigidity theorem} proven by Gromov in \cite{Gromov3}.

\begin{remark}
The analogous foliated Plateau problem for minimal surfaces has a solution for perturbations $h$ of the hyperbolic metric, as well as a slightly larger neighborhood of the hyperbolic metric subject to strong curvature conditions (see, \cite{Gromov_FolPlateau1,LoweGAFA}). However, such foliations do not always exist. Indeed, in \cite{LoweGAFA}, the second author constructs negatively-curved metrics $h$ for which $S_h\tilde{X}$ admits no foliation by lifts of minimal surfaces equivalent to the totally geodesic foliation of the unit tangent bundle of the hyperbolic metric by a leaf-preserving homeomorphism. It remains an open problem to determine necessary and sufficient conditions under which the foliated Plateau problem for minimal surfaces admits a solution.
\end{remark}

\begin{remark}We note that the above laminated spaces are all sublaminations of the laminated space introduced by Labourie in \cite{LabourieGAFA,LabourieInvent,LabourieAnnals} (see also \cite{Smith_asymp}). These sublaminations arise naturally in the study of quasi-Fuchsian $k$-surfaces.
\end{remark}

\subsubsection{The homogeneous case}
In the case where $h=h_0$ is the hyperbolic metric, the foliation $F\cF_{h_0,k}$ of Theorem \ref{th_foliated_Plateau} and the action $\alpha_k$ described in  \S \ref{sssec:LaminationsSimultaneousUniformization} are smoothly conjugated to a straightforward homogeneous model which we now describe. Note first that the natural action of $\Isom^+(\Hyp)\simeq\PSLc$ on the frame bundle $F_{h_0}\tilde{X}$ is free and transitive. Thus, up to a choice of base-frame, $F_{h_0}\tilde{X}\simeq\PSLc$, and right-multiplication defines a natural, free right-action of $\PSL$ on $F_{h_0}\tilde{X}$, which we denote by $\alpha_0$. Note that the orbit of any point $\xi:=(v_p,e_1,e_2)$ under this action is simply the frame bundle $FP$ of the unique, oriented, totally geodesic plane $P$ which passes through $p$ and has outward-pointing normal $v$ at this point. In particular, upon identifying $F_{h_0}\tilde{X}$ with the set of uniformizing maps of oriented, totally geodesic planes of $(\tilde{X},h_0)$, $\alpha_0$ may also be constructed exactly as in \S \ref{sssec:LaminationsSimultaneousUniformization}.

The right action $\alpha_0$ is smoothly conjugated to the right action $\alpha_k$ as follows. Let $FP$ be the frame bundle of an oriented totally geodesic plane $P$. Let $k:=\tanh^2(R)$ and consider the time $R$ map of the \emph{frame flow} $\hat G_R:F_{h_0}\tilde X\to F_{h_0}\tilde X,(p,v,e_1,e_2)\mapsto(G_R(p,v),e_1',e_2')$, where $G_R$ denotes the time $R$ map of the geodesic flow, and $(e_1',e_2')$ denotes the parallel transport of $(e_1,e_2)$ along this flow. The image of the frame bundle $FP$ of the totally geodesic plane $P$ is the frame bundle $FD$ of the $k$-disk $D$ spanned by $\bord P$, and $\hat G_R$ maps $P$ conformally onto $D$. It follows that $\hat G_R$ conjugates $\alpha_0$ and $\alpha_k$. We thus have the following identification.

\begin{theorem}[The homogeneous model]\label{t.homogeneous_model}
The right action $\alpha_k$ of $\PSL$ on $F_{h_0}\tilde{X}$ is smoothly conjugated to the homogeneous action of $\PSL$ on $\PSLc$ by multiplication on the right.
\end{theorem}

Since $p:F_{h_0}\tilde{X}\rightarrow S_{h_0}\tilde{X}$ is a principal $\text{SO}(2)$-bundle, $S_{h_0}\tilde{X}$ likewise identifies with the quotient space $\PSLc/\text{SO}(2)$. The images of the orbits of the homogeneous $\PSL$-action under this identification are then the Gauss lifts of totally umbilic planes. It follows that the $k$-surface foliation of $S_{h_0}\tilde{X}$ is likewise smoothly conjugate to the homogeneous foliation of $\PSLc/\text{SO}(2)$ by $\PSL$-orbits.

\subsection{Quasi-Fuchsian surfaces and topological counting}

\subsubsection{Quasi-Fuchsian surfaces} Given $C\geq 1$, we say that a surface subgroup $\Gamma\leq\Pi$ is $C$-\emph{quasi-Fuchsian} whenever $\partial_\infty\Gamma\dans\C\PP^1$ is a $C$-quasicircle. Let $\QF(C)$ denote the set of conjugacy classes of $C$-quasi-Fuchsian surface subgroups of $\Pi$. For any such conjugacy class $[\Gamma]$, we define the $k$-surface $S_{k,h}([\Gamma])$ by
\begin{equation}\label{eq.rep_surface}
S_{k,h}([\Gamma]):=\rD_{k,h}(\partial_\infty\Gamma)/\Gamma\ ,
\end{equation}
where $\Gamma$ is some representative of $[\Gamma]$. Note that the quotient of this surface under the action of $\Gamma$ is a closed $k$-surface in $(X,h)$ whose fundamental group identifies with $\Gamma$. Note, furthermore, that this quotient is independent of the representative chosen. We thus call $S_{k,h}([\Gamma])$ the $k$-surface \emph{representing} $[\Gamma$].

\subsubsection{Topological counting of quasi-Fuchsian surfaces} It follows from a theorem of Thurston (see, for example, \cite[Theorem 4.5]{LabourieBourbaki} for a minimal-surfaces based proof) that the set of conjugacy classes of quasi-Fuchsian subgroups $\Gamma\leq\Pi$ of genus $\text{g}(\Gamma)$ less than some fixed $g$ is finite. In \cite{KM2} (see also \cite{CMN}), Kahn--Markovi\'c provide a precise asymptotic estimate of the number of such conjugacy classes as a function of their genus, showing, in particular, that such classes are abundant.

\begin{theorem}[Kahn--Markovi\'c's topological counting \cite{KM2} (see also \cite{CMN})]\label{th.KM2}
If $(X,h_0)$ is a closed hyperbolic $3$-manifold then, for every $C\geq 1$,
\begin{equation}\label{eq.topo_counting}\lim_{g\to\infty}\frac{1}{2g\log(g)}\log\#\{[\Gamma]\in\mathrm{QF}(C);\,\mathrm{g}(\Gamma)\leq g\}=1.
\end{equation}
\end{theorem}

\begin{remark}To be precise, Kahn--Markovi\'c prove the above result only for $C=\infty$. In \cite[Theorem 4.2.]{CMN}, Calegari--Marques--Neves use the formula proven by M\"uller--Puchta in Lemma \ref{l.Muller_Puchta} to show that the above estimate also holds for all $C>1$.
\end{remark}

\subsection{Equidistribution of $k$-surfaces and $\PSL$-invariant measures}

\subsubsection{Conformal currents, $\PSL$-invariant measures and laminar measures}\label{sss.conf_currents_PSL_inv_laminar_measures} The terminology of fibered Plateau problems introduced in \S \ref{sss.PSL_quasicircles} allows us to generalize to higher dimensions the identification between geodesic currents and invariant measures of the geodesic flow studied, for example, by Bonahon in \cite{Bonahon_laminations}. The following construction was first suggested by Labourie in \cite{LabourieBourbaki}, and further developed by the present authors in \cite{ALS} (see also the recent works \cite{Brody_Reyes,MN_currents}).

Consider the bundles $\bord:\text{FKD}_{k,h}(C)\to\text{QC}^+(C)$ and $\bord:\text{MKD}_{k,h}(C)\to\text{QC}^+(C)$ defined in \S \ref{sss.PSL_quasicircles}. Recall that the space $\QC^+(C)$ is separable and locally compact. We say that a Borel regular measure $\mu$ on $\FKD_{k,h}(C)$ (resp. $\MKD_{k,h}(C)$) is $\PSL$-invariant whenever there exists a Borel regular measure $m$ on $\QC^+(C)$ such that, in every trivializing chart $U\times F$ of $\bord$,
\begin{equation*}
\mu|_{U\times F}=m|_U\times\lambda_F\ ,
\end{equation*}
where $\lambda_F$ denotes the Haar measure on $\PSL$ (resp. the hyperbolic area of $\D$). Borel regular $\PSL$-invariant measures on $\FKD(C)$ (resp. $\MKD(C)$) are trivially in one-to-one correspondence with Borel regular measures on $\QC^+$ (see \cite[Lemma 4.2.1.]{ALS}). We call $m$ the \emph{factor} of $\mu$, and we call $\mu$ the \emph{lift} of $m$. By construction, a $\PSL$-invariant measure $\mu$ is $\Pi$-invariant if and only if its factor $m$ is.

We thus have natural bijective correspondences between each of the following classes of objects (see \cite[\S 4.2]{ALS} for a slightly different formalism).
\begin{enumerate}
\item $\Pi$-invariant Borel regular measures $\hat m$ on the space $\text{QC}^+(C)$ of oriented $C$-quasicircles; \emph{these are the analogues in our context of geodesic currents}.
\item $(\Pi,\PSL)$ bi-invariant Borel measures $\hat\nu$ over $\FKD_{k,h}(C)$; \emph{these are the analogues in our context of measures invariant under the geodesic flow}.
\item $\Pi$-invariant Borel measures $\hat\mu$ on $\text{MKD}_{k,h}(C)$ which are totally invariant for the lamination $\cL_{k,h}$, that is, whose disintegration on each leaf of $\cL_{k,h}$ is a multiple of the hyperbolic area; \emph{these are the analogues in our context of totally invariant measures of the geodesic foliation}.
\end{enumerate}

According to Labourie's terminology (see \cite{LabourieBourbaki,MN_currents}), measures of the first class are called \emph{conformal currents}, measures of the second are called \emph{$(\Pi,\PSL)$-bi-invariant measures}, and measures of the third are called \emph{laminar measures}. This is summarized in the following table.

\begin{center}
\begin{tabularx}{\textwidth}{|>{\raggedright\arraybackslash}p{5cm}
                              |>{\raggedright\arraybackslash}X
                              |>{\raggedright\arraybackslash}X|}
    \hline
    \small\textbf{Measure type} &
    \small\textbf{Geodesic flow analogue} &
    \small\textbf{AKA. (see \cite{LabourieBourbaki,MN_currents})} \\
    \hline
    \small $\Pi$ invariant over $\text{QC}^+(C)$ &
    \small geodesic currents &
    \small conformal currents \\
    \small $(\Pi,\PSL)$ bi-invariant over $\FKD_{k,h}(C)$ &
    \small invariant measures of geodesic flow &
    \small bi-invariant measures \\
    \small $\Pi$ invariant over $\text{MKD}_{k,h}(C)$ and invariant wrt.\ $\cL_{k,h}$ &
    \small invariant measures of geodesic foliation &
    \small laminar measures \\
    \hline
\end{tabularx}
\end{center}

\begin{defi}[Ergodic conformal currents]
A conformal current $\hat m$ in $\QC^+(C)$ is said to be \emph{ergodic} if every $\Pi$-invariant measurable function on $\QC^+(C)$ is constant $\hat m$-almost everywhere.
\end{defi}

\subsubsection{Lifts of closed quasi-Fuchsian surfaces}\label{sss.lifts} Choose $C\geq 1$, let $[\Gamma]$ be a conjugacy class in $\text{QF}(C)$ and let $S:=S_{k,h}([\Gamma])$. We define the measure $\mu_S$ on the sphere bundle $S_h X$ such that, for every Borel subset $E$ of $S_h X$,
\begin{equation}\label{eq.inv_measure}
\mu_S(E)=\frac{1}{2\pi|\chi(S)|}\lambda(\pi(E\cap\hat S))\ ,
\end{equation}
where $\hat S$ denotes the Gauss lift of $S$, $\lambda$ denotes the area form of the Poincar\'e metric of $S$, and $\pi:S_hX\rightarrow X$ denotes the canonical projection.\\

Denote $c:=\bord\Gamma$, and let $\delta(c)$ denote the Dirac mass on $\text{QC}^+(C)$ supported on $c$. We now introduce the following conformal measure, bi-invariant measure, and laminated measure.\\

\noindent\emph{The conformal current associated to $S$}. We define
\begin{equation*}
\hat m(c)=\frac{1}{2\pi|\chi(S)|}\sum_{\gamma\in\Pi/\Gamma}\delta(\gamma\cdot c)\ ,
\end{equation*}
where $\gamma\cdot c$ abusively denotes the image of $c$ under any representative of $\gamma\in\Pi/\Gamma$. By construction, $m(c)$ is infinite, atomic and $\Pi$-invariant. We call it the \emph{conformal current} associated to $c$.\\

\noindent\emph{The bi-invariant measure associated to $S$}. Recall that the frame bundle $FD_c$ of $D_c$ naturally identifies, up to a choice of base point, with $\PSL$. Let $\omega_{k,h}(c)$ denote the image of the Haar measure under this identification. We define the $(\Pi,\PSL)$-bi-invariant measure associated to $S$ over $\FKD_{k,h}(C)$ by
\begin{equation*}
\hat\nu_{k,h}(c)=\frac{1}{2\pi|\chi(S)|}\sum_{\gamma\in\Pi/\Gamma}\omega_{k,h}(\gamma.c)\ .
\end{equation*}
This measure is analogous to the lift to $S_h \tilde X$ of the Dirac mass supported by a periodic orbit of the geodesic flow.\\

\noindent \emph{The laminar measure associated to $S$.} Let $\lambda_{k,h}(c)$ denote the area measure of the Poincar\'e metric of the disk $D_c$. The laminar measure associated to $S$ is the measure
\begin{equation*}
\hat\mu_{k,h}(c)=\frac{1}{2\pi|\chi(S)|}\sum_{\gamma\in\Pi/\Gamma}\lambda_{k,h}(\gamma.c)\ .
\end{equation*}
The quotient of this measure under the action of $\Pi$ yields a Borel regular measure $\mu_{k,h}$ over $\MKD_{k,h}(C)/\Pi$ supported on $\hat S$, where here $\hat S$ is viewed as a closed leaf of $\cL_{k,h}$. We identify the measure $\mu_{k,h}$ with $\mu_S$ as follows. Note first that leaves of $\cL_{k,h}$ are identified with Gauss lifts of $k$-surfaces of $X$. It follows that any continuous function $\phi:\MKD_{k,h}(C)/\Pi\to\R$ restricts to a continuous function $\phi:\hat S\to\R$. We then have
\begin{equation}\label{eq.ident_measures}
\int_{\MKD_{k,h}(C)/\Pi} \phi d\mu_{k,h}(c)=\int_{\hat S}\phi d\mu_S\ .
\end{equation}
Note that the conformal current depends neither on the metric $h$, nor on the constant $k$, but that bi-invariant measures and laminar measures do depend on these parameters.\\

\noindent \emph{The area measure.} Finally, we define the measure $\overline{\mu}_{S,h}$ on $S_hX$ by
\begin{equation}\label{eq.area_measure}
\overline{\mu}_{S,h}(E)=\frac{1}{2\pi|\chi(S)|}\Area_h(\pi(E\cap\hat S))\ ,
\end{equation}
where $\Area_h$ denotes the area measure of the restriction to $S$ of the metric $h$, and $\pi:S_hX\rightarrow X$ again denotes the canonical projection. The relationship between the measures $\mu_S$ and $\overline\mu_{S,h}$ is given by the following result.

\begin{lemma}[\cite{ALS}, Lemma 4.2.2]\label{lem_comparison_area_laminar}
If $-b\leq \sect_h\leq -a<0$ over $X$, and if $0<k<a$, then, for every closed, quasi-Fuchsian $k$-surface $S\dans X$,
\begin{equation*}
(a-k)\overline{\mu}_{S,h}\leq\mu_S\leq(b-k)\overline{\mu}_{S,h}.
\end{equation*}
\end{lemma}

\subsubsection{Classification of ergodic $(\Pi,\PSL)$-bi-invariant measures}

Following the main idea of \cite{CMN} (see also \cite{LabourieBourbaki}), we use Ratner's theory to study limits of measures of the form $\hat\mu_{k,h}(c)$ as $c$ tends to some round circle $c_0$. Note that the space $\cC^+$ of round circles in $\CP$ is a $3$-dimensional manifold carrying a natural Haar measure which we denote by $\Leb$.

\begin{theorem}[Classification of ergodic conformal currents]
\label{thm.Ratner} For every ergodic conformal current $\hat m$ on $\cC^+$, the following dichotomy holds.
\begin{itemize}
\item Either $\hat m$ is proportional to $\Leb$.
\item Or $\hat m$ is the conformal current of a closed Fuchsian surface.
\end{itemize}
\end{theorem}

\begin{proof}
We work with the hyperbolic metric $h:=h_0$ on $\tilde{X}$. Let $\hat\nu_0$ be an ergodic measure over $\PSLc\simeq F_{h_0}\tilde{X}$ which is both invariant under the left action of $\Pi$, and under the right action by multiplication of $\PSL$. By Ratner's classification theorem \cite{Ratner_Duke} (see \cite{Einsiedler} for a simple proof for the case of $\PSL$-actions), this measure is \emph{homogeneous}. That is, it is the Haar measure supported on some closed orbit $v_0G$ of some closed, connected Lie subgroup $G\leq\PSLc$ containing $\PSL$. Since the only closed and connected Lie subgroups of $\PSLc$ containing $\PSL$ are $\PSL$ and $\PSLc$ itself, this yields the following dichotomy:
\begin{itemize}
\item either $\nu_0$ is the Haar measure of $F_{h_0}\tilde X/\Pi$;
\item or $\nu_0$ is the Haar measure of the frame bundle of some closed, totally-umbilic surface of $\tilde X/\Pi$.
\end{itemize}
The result now follows.
\end{proof}

Theorem \ref{thm.Ratner} also yields the following dichotomy for ergodic laminar measures over $S_hX$.

\begin{theorem}[Dichotomy for laminar measures]\label{thm.dichotomy_laminar} There exists a unique ergodic laminar probability measure $\mu$ on $S_hX$ with full support. All other ergodic laminar probability measures on $S_hX$ are associated to closed Fuchsian $k$-surfaces.
\end{theorem}

\begin{remark}
Measures which arise as limits of random minimal surfaces are studied by Kahn--Markovi\'c--Smilga in \cite{KMS}. We expect analogous results to hold for limits of random $k$-surfaces.
\end{remark}

\begin{remark}
When $\Pi$ contains no Fuchsian subgroup, there exists a unique laminar probability measure on $S_hX$ which is fully supported.
\end{remark}

\begin{proof}
Let $\hat\mu$ denote the lift of $\mu$ to $S_h\tilde{X}$ and let $\hat m$ denote its associated conformal current on $\cC^+$. If $\mu$ has full support, then so too does $\hat m$ so that, by Theorem \ref{thm.Ratner}, $\hat m$ is proportional to $\Leb$, and the first assertion follows. Likewise, if $\mu$ does not have full support, then, by Theorem \ref{thm.Ratner} again, $\hat\mu$ is the laminar measure associated to some closed Fuchsian $k$-surface, as in \S \ref{sss.lifts}, and the second assertion follows.
\end{proof}

\subsubsection{Equidistribution and Kahn-Markovi\'c's sequences} The following facts are key to proving the main results of \cite{ALS}.

\begin{lemma}\label{lem.conv_fuchsian}
Let $(\eps_n)_{n\in\mathbb{N}}$ be a sequence of positive numbers converging to zero. For all $n$, let $\Gamma_n$ be a $(1+\eps_n)$-quasi-Fuchsian subgroup of $\Pi$, denote $c_n:=\bord\Gamma_n$ and let $S_n=S_{k,h}([\Gamma_n])$ denote the $k$-surface representative of $\Gamma_n$ in $(X,h)$. If $\lim_{n\to\infty}\eps_n=0$, then every accumulation point of $(\hat\mu_{k,h}(c_n))_{n\in\mathbb{N}}$ is a laminar measure on $\MKD_{k,h}(1)\simeq S_{h}\tilde{X}$.
\end{lemma}

Kahn--Markovi\'c's construction \cite{KM1} of quasi-Fuchsian subgroups of $\Pi$ yields a sequence of surfaces with remarkable properties.

\begin{theorem}[Equidistribution of Kahn-Markovi\'c sequences]\label{KM_equidistrib}
There exists a sequence $(\eps_n)_{n\in\mathbb{N}}$ of positive numbers converging to zero, and a sequence $(\Gamma_n)_{n\in\mathbb{N}}$ of quasi-Fuchsian subgroups of $\Pi$ satisfying the following properties.
\begin{enumerate}
\item For all $n$, $\Gamma_n$ is $(1+\eps_n)$-quasi-Fuchsian; and
\item the sequence of laminar measures $(\hat \mu_{k,h}(\bord\Gamma_n))_{n\in\mathbb{N}}$ converges to a laminar measure $\hat \mu_\infty$ on $S_h\tilde{X}$ with full support.
\end{enumerate}
\end{theorem}

We provide a proof of this theorem in \cite{ALS} by showing that the $k$-surfaces representing the subgroups constructed by Kahn--Markovi\'c in \cite{KM1} do not accumulate on closed surfaces, from which Theorem \ref{KM_equidistrib} follows as a consequence of Theorem \ref{thm.dichotomy_laminar}. In \cite[Proposition 6.1]{ln21}, in collaboration with A. Neves, the second author proves the following stronger result.

\begin{theorem}[Lowe--Neves]\label{th.lowe_neves}
There exists a sequence $(\eps_n)_{n\in\mathbb{N}}$ of positive numbers converging to zero, and a sequence $(\Gamma_n)_{n\in\mathbb{N}}$ of quasi-Fuchsian subgroups of $\Pi$ satisfying the following properties.
\begin{enumerate}
\item For all $n$, $\Gamma_n$ is $(1+\eps_n)$-quasi-Fuchsian; and
\item the sequence of laminar measures $(\hat \mu_{k,h}(\bord\Gamma_n))_{n\in\mathbb{N}}$ converges to a fully supported and ergodic laminar measure on $S_h\tilde{X}$.
\end{enumerate}
\end{theorem}

\section{Fuchsian, almost Fuchsian and totally umbilical $k$-surfaces} In this section we study some geometric properties of $k$-surfaces in closed $3$-manifolds of negative sectional curvature bounded below by $-1$. We define the $p$-energy of such a surface, and we prove a $k$-surface analogue of the result \cite{Seppi} of Seppi, namely that almost-Fuchsian $k$-surfaces in hyperbolic $3$-space are almost totally umbilical.

\subsection{$p$-energy of surfaces} Recall that $(X,h)$ is a closed $3$-dimensional riemannian manifold of sectional curvature pinched between $-b$ and $-a$, where $0<a<b$. Given a closed, oriented immersed surface $S\dans X$, and a real number $p\geq 0$, we define the \emph{$p$-energy} of $S$ by
\begin{equation*}
W_h^{p}(S):=\int_S H^p d\Area_h\ ,
\end{equation*}
and we define its \emph{normalized $p$-energy} by
\begin{equation*}
\overline{W}_h^{p}(S):=k^{-p/2}W_h^p(S)=k^{-p/2}\int_S H^p d\Area_h\ .
\end{equation*}

\begin{lemma}\label{lem_lower_bounds_energy}
If $-1\leq\sect_h\leq -a$, $0<k<a$, $p\geq 0$, then, for every closed $k$-surface $S$,
\begin{enumerate}
\item $\Area_h(S)\geq 2\pi|\chi(S)|/(1-k)$ with equality holding if and only if the sectional curvature of $h$ equals $-1$ along $TS$;
\item $W_h^{p}(S)\geq k^{p/2}\Area_h(S)$ with equality holding if and only if $S$ is totally umbilical.
\end{enumerate}
\end{lemma}

\begin{proof}
Indeed, by Gauss' equation and the Gauss--Bonnet theorem,
\begin{equation*}
2 \pi \chi(S) = \int_{S} \kappa_{int}d\Area_h \geq \int_{S} (-1 + k)d\Area_h=\Area_h(S) (k-1)\ .
\end{equation*}
Hence
\begin{equation*}
\Area_h(S) \geq \frac{2\pi\left|\chi(S)\right|}{(1-k)}\ ,
\end{equation*}
with equality holding if and only if $\sect_h=-1$ over $TS$. This proves the first assertion.

Since the surfaces studied here are oriented in such a way that their principal curvatures are always positive, the algebraic-geometric mean inequality yields
\begin{equation*}
H\geq k^{\frac{1}{2}}\ ,
\end{equation*}
so that
\begin{equation*}
W_h^p(S) \geq k^{p/2}\Area_h(S)\ .
\end{equation*}
Furthermore, equality holds if and only if $H$ is constant and equal to $k^{\frac{1}{2}}$, that is, if and only if $S$ is totally umbilical, and this proves the second assertion.
\end{proof}

\subsection{Almost totally umbilical $k$-disks in hyperbolic space}\label{ss.totally_umb} Consider first the case where $h:=h_0$ is the hyperbolic metric on $\tilde{X}$. Fix $k\in(0,1)$ and $p>0$. We show that, as $\eps$ tends to $0$, $k$-disks in $(\tilde{X},h_0)$ spanned by $(1+\eps)$-quasicircles become uniformly close to totally umbilical surfaces in the sense that their principal curvatures become uniformally close to $k^{\frac{1}{2}}$. This is the $k$-surface analogue of the result \cite{Seppi} of Seppi concerning minimal surfaces spanned by almost round quasicircles. However, in the present case, the proof is simpler.

\begin{lemma}\label{l.almost_tot_umbilical}
For all $\delta>0$ there exists $\eps>0$ having the property that, for every oriented $(1+\eps)$-quasicircle $c$, if $H$ denotes the mean curvature of the unique $k$-disk in $(\tilde{X},h_0)$ spanned by $c$, then
\begin{equation*}
k^{\frac{1}{2}}\leq H\leq k^{\frac{1}{2}}(1+\delta)\ .
\end{equation*}
\end{lemma}

\begin{proof}
The first inequality follows immediately from the algebraic-geometric mean inequality.

Suppose now that the second inequality does not hold. There exists $\delta>0$, a sequence $(\eps_n)_{n\in\N}$ converging to $0$, and a sequence $(D_n,p_n)_{n\in\N}$ of marked $k$-disks such that, for all $n$, the curve $c_n:=\bord D_n$ is a $(1+\eps_n)$-quasicircle, and the mean curvature $H_n$ of $D_n$ satisfies
\begin{equation}\label{eq.umbilical}
\liminf_{n\to\infty} H_n(p_n)>(1+\delta)k^{\frac{1}{2}}\ .
\end{equation}
There exists a sequence $(\gamma_n)_{n\in\N}$ of isometries of $(\tilde{X},h_0)$ such that, for all $n$, $\gamma_n\cdot p_n$ lies within some fixed compact set. Upon extracting a subsequence, we may suppose that the sequence $(\gamma_n\cdot c_n)_{n\in\N}$ of quasicircles converges in the Hausdorff sense, either to a round circle $c_\infty$, say, or to a single ideal point $q_\infty$, say. The latter is impossible, for otherwise the sequence $(\gamma_n\cdot D_n)_{n\in\N}$ would converge to a $k$-surface in $(\tilde{X},h_0)$ with ideal boundary $\{q_\infty\}$, which is absurd by \cite[Proposition 2.5.3]{LabourieInvent}. By Theorem \ref{th.As_plateau_problem}, the sequence $(\gamma_n\cdot D_n)_{n\in\N}$ then converges smoothly to a $k$-disk $D_\infty$, say, of $(\tilde{X},h_0)$ whose ideal boundary is the round circle $c_\infty$. In particular, $D_\infty$ is totally umbilical, so that its mean curvature is everywhere equal to $k^{\frac{1}{2}}$ (see \S \ref{sss.FPP}), contradicting \eqref{eq.umbilical}. This is absurd, and the result follows.
\end{proof}

By Gauss' equation, every $k$-surface of $(\tilde{X},h_0)$ has intrinsic curvature equal to $(k-1)$. Upon combining Lemma \ref{l.almost_tot_umbilical} with the Gauss--Bonnet theorem, we therefore obtain the following result.

\begin{corollary}\label{corollarouille}
For all $0<k<1$, for all $p>0$, for every sequence $(\eps_n)_{m\in\N}$ of positive numbers that tends to zero, and every sequence $(S_n)_{n\in\N}$ of closed  $(1+\eps_n)$-quasi-Fuchsian surfaces in $(X,h_0)$,
\begin{equation*}
\lim_{n\to\infty}\frac{k^{-p/2}}{\Area_{h_0}(S_n)}\int_{S_n} H_n^p d\Area_{h_0}=1\ ,
\end{equation*}
and
\begin{equation*}
\lim_{n\to\infty}\frac{(1-k)k^{-p/2}}{2\pi|\chi(S_n)|}\int_{S_n} H_n^p d\Area_{h_0}=1\ .
\end{equation*}
That is, the marked normalized $p$-energy spectrum of $k$-surfaces in $(\tilde{X},h_0)$ is asymptotic to the marked area spectrum of $k$-surfaces in this space.
\end{corollary}

\section{The marked energy spectrum}

\subsection{Proof of Theorem \ref{thm.Rigidity_hyp_energy_spectrum}}\label{ss.Rigidity}
Now let $h$ be any riemannian metric on $X$ with sectional curvature pinched between $-1$ and $-a$ for some $0<a<1$. Suppose, furthermore, that the marked $p$-energy spectrum of $k$-surfaces of $h$ is equal to that of $h_0$, that is, that equality holds in \eqref{eqn.rigidity_energy}. We will prove that $h$ is isometric to $h_0$. By Mostow's rigidity theorem \cite{Mostow}, it suffices to show that $h$ has constant sectional curvature equal to $-1$.

\subsubsection{Sectional curvature approaches $-1$ over asymptotically Fuchsian $k$-surfaces}Let $([\Gamma_n])_{n\in\N}\in(\QF^+)^\N$ be a sequence of conjugacy classes of quasi-Fuchsian surface subgroups of $\Pi$. For all $n$, let $c_n$ denote the ideal boundary of $\Gamma_n$. Suppose that there exists a sequence $(\eps_n)_{n\in\N}$ converging to $0$ such that, for all $n$, $c_n$ is a $(1+\eps_n)$-quasicircle. Suppose, furthermore, that $(c_n)_{n\in\N}$ converges in the Hausdorff sense to some round circle $c_\infty$, say. For all $n$, denote
\begin{equation*}
S_n=\rD_{k,h}(c_n)/\Gamma_n\dans X\ ,
\end{equation*}
and let $H_n$ denote the mean curvature of this surface. Let $\sigma:S_hX\to\R$ denote the map which associates to every $v\in S_hX$ the sectional curvature of $h$ along the plane $v^\perp$.

Recall the definition of the area measure $\overline{\mu}_{S_n,h}$ from \S \ref{sss.lifts}.

\begin{lemma}\label{l.sect_approaches-1}
\begin{equation*}
\lim_{n\to\infty} \int_{S_hX}(1+\sigma)d\overline{\mu}_{S_n,h}=0\ .
\end{equation*}
\end{lemma}

\begin{remark}
Note that we use here the fact that the sectional curvature of $h$ is bounded below by $-1$.
\end{remark}

\begin{proof}
Indeed, for all $n$, denote
\begin{equation*}
S_n^0=\text D_{k,h_0}(c_n)/\Gamma_n\ ,
\end{equation*}
and let $H_n^0$ denote the mean curvature of this surface with respect to the hyperbolic metric $h_0$.

By hypothesis, the marked $p$-energy spectra of $h$ and $h_0$ coincide so that, for all $n$,
\begin{equation*}
\int_{S^0_n}(H_n^0)^p d\Area_{h_0}=\int_{S_n} H_n^p d\Area_{h}\ ,
\end{equation*}
and it follows by Corollary \ref{corollarouille} that
\begin{equation}\label{eq.rigidity_marked_spectra}
\lim_{n\to\infty}\frac{1}{2\pi|\chi(S_n)|}\int_{S_n} H_n^p d\Area_{h}=\frac{k^{p/2}}{(1-k)}\ .
\end{equation}
On the other hand, the algebraic-geometric mean inequality yields
\begin{eqnarray*}
\int_{S_n} H_n^p d\Area_{h}&\geq&\int_{S_n} k^{p/2}d\Area_{h}\\
&=&\frac{k^{p/2}}{(1-k)}\bigg(\int_{S_n}(1+\sect_h|_{TS_n})d\Area_{h}\\
& &\qquad\qquad + \int_{S_n} (-\sect_h|_{TS_n} - k)d\Area_{h}\bigg)\ .
\end{eqnarray*}
By Gauss' equation, for all $n$, the intrinsic curvature of $S_n$ is equal to $k+\sect_h|_{TS_n}$, so that, by the Gauss--Bonnet theorem,
\begin{equation*}
\int_{S_n} H_n^p d\Area_{h} \geq \frac{k^{p/2}}{(1-k)}\bigg(\int_{S_n}(1+\sect_h|_{TS_n})d\Area_{h} - 2\pi\chi(S_n)\bigg)\ ,
\end{equation*}
and dividing by $2\pi\left|\chi(S_n)\right|$ yields
\begin{equation}\label{eq.GB}
\frac{1}{2\pi|\chi(S_n)|}\int_{S_n} H_n^p d\Area_{h}\geq \frac{k^{p/2}}{(1-k)}\bigg(\int_{S_hX}(1+\sigma)d\overline\mu_{S_n,h}+1\bigg)\ .
\end{equation}
Taking limits in \eqref{eq.GB}, and bearing in mind \eqref{eq.rigidity_marked_spectra}, we therefore obtain
\begin{equation*}
\lim_{n\rightarrow\infty}\int_{S_hX}(1+\sigma)d\overline\mu_{S_n,h}\leq 0\ .
\end{equation*}
However, by hypothesis,
\begin{equation*}
(1+\sigma)\geq 0\ .
\end{equation*}
The above integral is therefore also non-negative, and the result follows.
\end{proof}

\subsubsection{Using the equidistribution} We now make use of the equidistribution result of Theorem \ref{KM_equidistrib}. Indeed, let $([\Gamma_n])_{n\in\N}$ be a sequence of conjugacy classes of oriented, quasi-Fuchsian subgroups of $\Pi$ as in Theorem \ref{KM_equidistrib}.
\begin{itemize}
\item By Lemma \ref{lem_comparison_area_laminar}, for all $n$, $(a-k)\overline\mu_{S_n,h}\leq\mu_{S_n}\leq(1-k)\overline\mu_{S_n,h}$;
\item By Theorem \ref{KM_equidistrib}, the weak-$\ast$ limit of the sequence $\mu_{S_n}$ is a probability measure $\mu_\infty$ on $S_hX$ of full support.
\end{itemize}
By Lemma \ref{l.sect_approaches-1},
\begin{equation*}
\int(1+\sigma)d\mu_\infty=0\ ,
\end{equation*}
and since $1+\sigma\geq 0$,
\begin{equation*}
\sigma=-1\qquad \mu_\infty\text{-a.e.}
\end{equation*} in $S_hX$. Since $\mu_\infty$ has full support, it follows that $\sigma=-1$ over $S_hM$. In other words, $h$ has sectional curvature everywhere equal to $-1$, so that, by Mostow's rigidity theorem, $h$ and $h_0$ are isometric.

\subsection{Proof of Theorem \ref{thm.asymptotic_spectra}}\label{ss.proof_of_thm_B}

Suppose now that the sectional curvature of $X$ is pinched between $-b$ and $-a$ for some $0<a\leq b$. We first assume that the marked normalized energy spectrum is asymptotic to the marked area spectrum, that is
\begin{equation}\label{eq.hypothesis}
\lim_{n\to\infty}\frac{1}{\Area_h(S_n)}\int_{S_n}\big(k^{-\frac{1}{2}}H_n-1\big)d\Area_h=0\ ,
\end{equation}
for every sequence $(S_n)_{n \in\N}$ of $(1+\eps_n)-$quasi-Fuchsian surfaces such that $\eps_n\to 0$ as $n\to\infty$.

Given $C\geq 1$, let $H:\MKD_{k,h}(C)\to\R$ denote the map which associates to every $(D,p)\in \MKD_{k,h}(C)$ the mean curvature of the oriented $k$-disk $D$ at the point $p$. This map is obviously $\Pi$-invariant, and we denote the function that it defines over the quotient $\MKD_{k,h}(C)/\Pi$ also by $H$. When $C=1$, $H$ is the map which associates to each $v\in S_hX$ the mean curvature of the leaf of $\cF$ containing $v$. Theorem \ref{thm.asymptotic_spectra} will follow from the following result.

\begin{prop}\label{p.tot_asymptotic}
$H=k^{\frac{1}{2}}$ over $S_hX$.
\end{prop}

\begin{proof}[Proposition \ref{p.tot_asymptotic} implies Theorem \ref{thm.asymptotic_spectra}]
Indeed, let $\hat L$ be a leaf of $\cF_{k,h}$, and let $L$ denote its projection in $S_hX$. Let $v$ be a point of $\hat L$, and let $p$ denote its base point. If $H=k^{\frac{1}{2}}$ at $v$ then both principal curvatures of $L$ at $p$ coincide, that is $p$ is an umbilic point of $L$. Proposition \ref{p.tot_asymptotic} therefore implies that the projection to $X$ of every leaf of $\cF$ is totally umbilical and has constant mean curvature. The manifold $X$ therefore satisfies the \emph{axiom of $2$-spheres} of Leung--Nomizu (see \cite{Leung_Nomizu}), namely \emph{for every point $p$ of $X$ and every plane $P$ of $T_pX$, there exists a totally umbilical, constant mean curvature surface $S$ tangent to $P$ at $p$.} It is proven in \cite{Leung_Nomizu} that $X$ has constant sectional curvature whenever this property is satisfied, and Theorem \ref{thm.asymptotic_spectra} therefore follows.
\end{proof}

\begin{proof}[Proof of Proposition \ref{p.tot_asymptotic}] We apply \eqref{eq.hypothesis} to the Kahn-Markovi\'c sequence of Theorem \ref{KM_equidistrib}. For all $n$, let $S_n,\Gamma_n$ and $c_n$ be as in that theorem. Let $\lambda_{S_n}$ denote the area measure \emph{for the hyperbolic metric} in $S_n$. By \eqref{eq.ident_measures} and the Gauss-Bonnet theorem,
\begin{equation*}
\int_{\MKD_{k,h}(1+\eps_n)/\Pi}\big(k^{-\frac{1}{2}}H-1\big)d\mu_{k,h}(c_n)=
\frac{1}{\lambda_{S_n}(S_n)}\int_{S_n}\big(k^{-\frac{1}{2}}H-1\big)d\lambda_{S_n}\ .
\end{equation*}
Thus, by Lemma \ref{lem_comparison_area_laminar},
\begin{equation*}
\multiline{
\int_{\MKD_{k,h}(1+\eps_n)/\Pi}\big(k^{-\frac{1}{2}}H-1\big)d\mu_{k,h}(c_n)\cr
\qquad\qquad\qquad\qquad\leq\frac{(b-k)}{(a-k)}\frac{1}{\Area_h(S_n)}\int_{S_n}\big(k^{-\frac{1}{2}}H_n-1\big)d\Area_h\ .\cr}
\end{equation*}
By \eqref{eq.hypothesis}, the last integral converges to $0$. However, for the Kahn-Markovi\'c sequence, $(\mu_{k,h}(c_m))_{n\in\mathbb{N}}$ converges to a measure $\mu_\infty$, say, of full support over $\MKD_{k,h}(1)/\Pi\simeq S_hX$. Hence,
\begin{equation*}
\int_{S_hX}\big(k^{-\frac{1}{2}}H-1\big)d\mu_\infty = 0\ .
\end{equation*}
Since $H$ is continuous, and since $\mu_\infty$ has full support, it follows that $H=k^{\frac{1}{2}}$ everywhere, and this completes the proof.
\end{proof}

For the general case, it suffices to observe that the same argument also show that the sectional curvature is constant whenever the normalized marked $p$ and $q$ spectra are asymptotic. Indeed, for any closed surface $S$,
\begin{equation*}
\multiline{
\frac{\int_S k^{-\frac{p}{2}}H^p d\Area_h}{\int_S k^{-\frac{q}{2}}H^qd\Area_h}-1\cr
\qquad\qquad\qquad\qquad=\frac{1}{\int_Sk^{-\frac{q}{2}}H^qd\Area_h}\int_S\big(k^{-\frac{p-q}{2}}H^{p-q}-1\big)k^{-\frac{q}{2}} H^qd\Area_h\ ,\cr}
\end{equation*}
and the result follows as before.

\section{Asymptotic counting according to energy}
\subsection{Energy entropy for $k$-surfaces}
\subsubsection{The energy entropy} Given $p\geq 0$ we define
\begin{equation*}
\Ent^p_{k}(X,h)=\lim_{\eps\to 0}\liminf_{w\to\infty}\frac{1}{w\log w}\log\#\left\{[\Gamma]\in\QF(1+\eps)\ \bigg|\ W^{p}_{h}(S_{k,h}([\Gamma]))\leq w\right\}\ .
\end{equation*}
This quantity is analogous to the minimal surface entropy studied in \cite{CMN,ln21}. The \emph{area entropy} used in \cite{ALS} is more elementary since it does not involve a double limit. However, as in the case of the area entropy of minimal surfaces, it is hard to imagine how to compute the energy entropy without using such a double limit, even for the hyperbolic metric.

\begin{lemma}[Entropy in the hyperbolic case] If $h_0$ is the hyperbolic metric on $X$, then for all $p\geq 0$ and $0<k<1$,
\begin{equation*}
\Ent^p_{k}(X,h_0)=\frac{k^{-\frac{p}{2}}(1-k)}{2\pi}\ .
\end{equation*}
\end{lemma}

\begin{proof}
We first prove the upper bound
\begin{equation*}
\Ent^p_{k}(X,h_0)\leq\frac{k^{-\frac{p}{2}}(1-k)}{2\pi}\ .
\end{equation*}
Indeed, by Lemma \ref{lem_lower_bounds_energy}, for every $[\Gamma]\in\QF(1+\eps)$,
\begin{equation*}
W_{h_0}^{p}(S_{k,h_0}([\Gamma]))\geq k^{\frac{p}{2}}\Area_{h_0}(S_{k,h_0}([\Gamma]))=\frac{4\pi(g-1)k^{\frac{p}{2}}}{(1-k)}\ ,
\end{equation*}
where $g$ denotes the genus of $S_{k,h_0}([\Gamma])$. It follows that if
\begin{equation*}
W_{h_0}^{p}(S_{k,h_0}([\Gamma])) \leq w\ ,
\end{equation*}
then
\begin{equation*}
g \leq \frac{k^{-\frac{p}{2}}(1-k)w}{4\pi} + 1\ ,
\end{equation*}
and the upper bound now follows by Kahn-Markovi\'c's topological counting result (Theorem \ref{th.KM2}).

We now use Lemma \ref{l.almost_tot_umbilical} to bound the entropy from below
\begin{equation*}
\Ent^p_{k}(X,h_0)\geq\frac{k^{-\frac{p}{2}}(1-k)}{2\pi}\ .
\end{equation*}
Fix $\delta>0$ and let $\eps_0>0$ be as in Lemma \ref{l.almost_tot_umbilical}. For all $\eps<\eps_0$ and for all $[\Gamma]\in\QF(1+\eps)$,
\begin{equation*}
W^{p}_{h_0}(S_{k,h_0}([\Gamma])\leq (1+\delta)^p\frac{4\pi(g-1)k^{\frac{p}{2}}}{(1-k)}\ .
\end{equation*}
It follows that if
\begin{equation*}
g \leq \frac{(1-k)w k^{-{\frac{p}{2}}}}{4\pi(1+\delta)^p} + 1\ ,
\end{equation*}
then
\begin{equation*}
W_{h_0}^{p}(S_{k,h_0}([\Gamma])) \leq w\ ,
\end{equation*}
and, proceeding as before, we show that
\begin{equation*}
\Ent^{p}_{k}(X,h_0)\geq\frac{(1-k)k^{-\frac{p}{2}}}{2\pi(1+\delta)^p}\ .
\end{equation*}
Since $\delta$ is arbitrary, the result follows.
\end{proof}

Note that in order to obtain the upper bound in the previous proof, we only used the inequality $\sect_h\geq -1$. We therefore also have the following result.

\begin{lemma}\label{l_inequality_entropy}
Let $(X,h_0)$ be a closed, $3$-dimensional hyperbolic manifold. Let $h$ be another riemannian metric on $X$ of sectional curvature pinched between $-1$ and $-a$, for some $0<a<1$. For all $p\geq 0$, and for all $0<k<a$,
\begin{equation*}
\Ent^{p}_{k}(X,h)\leq\frac{k^{-\frac{p}{2}}(1-k)}{2\pi}=\Ent^{p}_{k}(X,h_0)\ .
\end{equation*}
\end{lemma}

\subsubsection{Entropy and closed Fuchsian $k$-surfaces} It is worth noting that the $k$-surface energy entropy counts closed \emph{conjugacy classes} of surface subgroups. This has the following consequence.

\begin{theorem}\label{th.totally_umb_implies_big_entropy}
Assume that $-1\leq\sect_h\leq-a$, $p\geq 0$ and $0<k<a$. Assume furthermore that there exists a closed \emph{Fuchsian} surface subgroup $\Gamma\leq\Pi$ such that
\begin{equation*}
W^{p}_{h}(S_{k,h}([\Gamma]))=W^{p}_{h_0}(S_{k,h_0}([\Gamma]))\ .
\end{equation*}
Then
\begin{equation*}
\Ent^{p}_{k}(X,h)=\Ent^{p}_{k}(X,h_0).
\end{equation*}
\end{theorem}

\begin{remark}
If $p>0$, one can check that $W^{p}_{h}(S_{k,h}([\Gamma]))=W^{p}_{h_0}(S_{k,h_0}([\Gamma]))$ if and only if $[\Gamma]$ is represented in $(M,h)$ by a totally umbilic surface with constant mean curvature $k^{1/2}$ (compare Question \ref{question:intro}.)
\end{remark}

The theorem relies on Lemma \ref{l.Muller_Puchta}, below, which is a consequence of Formula (3) of \cite{Muller_Puchta}. We first require the following definition.

\begin{defi}\label{d.conj_classes}
Let $\Gamma$ be a surface subgroup of $\Pi$. We denote by $\sigma(\Gamma)$ the set of $\Pi$-conjugacy classes of finite-index subgroups of $\Gamma$. For every $[\Gamma']\in\sigma(\Gamma)$, we denote by $\mathrm{g}(\Gamma')$ the genus of any subgroup $\Gamma'$ representing the class.
\end{defi}

\begin{lemma}[M\"uller--Puchta's lower bound]\label{l.Muller_Puchta}
For every compact surface subgroup $\Gamma$ of $\Pi$,
\begin{equation}\label{eq.lower_bound_MP}
\liminf_{g\rightarrow\infty}\frac{\log\big(\#\{[\Gamma']\in\sigma(\Gamma)\ |\ \mathrm{g}(\Gamma')\leq g\}\big)}{2g\log(g)} \geq 1\ .
\end{equation}
\end{lemma}

\begin{proof}
We provide a proof of this estimate since \cite{Muller_Puchta} in fact counts \emph{subgroups}, not \emph{conjugacy classes}. For convenience, we suppose that $\Gamma$ is the stabilizer in $\Pi$ of its boundary curve $\partial_\infty\Gamma$.

Upon applying Stirling's approximation to Equation (3) of \cite{Muller_Puchta}, we find that
\begin{equation*}
\liminf_{g\rightarrow\infty}\frac{\log\big(\#\{\Gamma'\subset\Gamma\ |\ \mathrm{g}(\Gamma')\leq g\}\big)}{2g\log(g)} \geq 1\ .
\end{equation*}
In terms of the Euler characteristic, this becomes
\begin{equation*}
\liminf_{g\rightarrow\infty}\frac{\log\big(\#\{\Gamma'\subset\Gamma\ |\ \left|\chi(\Gamma')\right|\leq x\}\big)}{x\log(x)} \geq 1\ .
\end{equation*}

Let $x_0:=\left|\chi(\Gamma)\right|$ denote the absolute value of the Euler characteristic of $\Gamma$. Let $\Gamma'$ be a finite-index subgroup of $\Gamma$ with
\begin{equation*}
\left|\chi(\Gamma')\right| \leq x\ .
\end{equation*}
Then $\Gamma'$ corresponds to an $N$-fold cover, where
\begin{equation*}
N\leq x/x_0\ .
\end{equation*}
We now claim that the number of $\Pi$-conjugates of $\Gamma'$ in $\Gamma$ is bounded above by $(x/x_0)$. Indeed, let $\Gamma''$ be a $\Pi$-conjugate of $\Gamma'$ in $\Gamma$. Let $\gamma\in\Pi$ be such that $\Gamma''=\gamma\Gamma'\gamma^{-1}$. Since $\Gamma'$ and $\Gamma''$ are both finite-index subgroups of $\Gamma$, $\gamma$ preserves $\partial_\infty\Gamma$, and is therefore an element of $\mathrm{Stab}(\partial_\infty\Gamma)=\Pi$. It follows that $\Gamma''$ is in fact a $\Gamma$-conjugate of $\Gamma'$.

Note now that the number of distinct $\Gamma$-conjugates of $\Gamma'$ in $\Gamma$ is bounded above by the index of $\Gamma'$ in $\Gamma$. Indeed, if $\gamma$ and $\mu$ belong to the same left coset of $\Gamma'$, then
\begin{equation*}
\gamma\Gamma'\gamma^{-1} = \mu\Gamma'\mu^{-1}\ .
\end{equation*}
Finally, by looking at fundamental domains, we see that this index is bounded above by $N\leq (x/x_0)$, and the assertion follows. 

It follows that
\begin{equation*}
\#\{[\Gamma']\in\sigma(\Gamma)\ |\ \left|\chi(\Gamma')\right| \leq x\}
\geq \#\{\Gamma'\subset\Gamma\ |\ \left|\chi(\Gamma')\right| \leq x\}/(x/x_0)\ .
\end{equation*}
Hence
\begin{equation*}
\liminf_{g\rightarrow\infty}\frac{\log\big(\#\{[\Gamma']\in\sigma(\Gamma)\ |\ \left|\chi(\Gamma')\right|\leq x\}\big)}{x\log(x)} \geq 1\ ,
\end{equation*}
and the result follows.
\end{proof}

\begin{proof}[Proof of Theorem \ref{th.totally_umb_implies_big_entropy}.] It suffices to prove that
\begin{equation*}
\Ent^{p}_{k}(X,h)\leq\Ent^{p}_{k}(X,h_0)\ ,
\end{equation*}
since the reverse inequality holds by Lemma \ref{l_inequality_entropy}.

Let $\Gamma$ be a Fuchsian surface subgroup of $\Pi$ and suppose that
\begin{equation*}
W^{p}_{h}(S_{k,h}([\Gamma]))=W^{p}_{h_0}(S_{k,h_0}([\Gamma]))=\frac{4\pi k^{\frac{p}{2}}}{(1-k)}\mathrm{g}(\Gamma)\ ,
\end{equation*}
where the last equality follows from Lemma \ref{lem_lower_bounds_energy}. By Lemma \ref{lem_lower_bounds_energy} again, the surface $S=S_{k,h}([\Gamma])$ must be totally umbilical with constant mean curvature equal to $k^{\frac{1}{2}}$.

Let $[\Gamma']$ be an element of $\sigma(\Gamma)$, as in Definition \ref{d.conj_classes}. Any subgroup $\Gamma'$ representing this class is quasi-Fuchsian with $\bord \Gamma=\bord \Gamma'$, so that $[\Gamma']\in\QF(1+\eps)$ for all $\eps>0$, and
\begin{equation*}
S_{k,h}([\Gamma'])=\mathrm{D}_{k,h}(\bord \Gamma)/\Gamma'\ .
\end{equation*}
In particular, $S_{k,h}([\Gamma'])$ is just a finite cover of $S_{k,h}([\Gamma])$, and is therefore also totally umbilical with constant mean curvature equal to $k^{\frac{1}{2}}$. Consequently, by Lemma \ref{lem_lower_bounds_energy},
\begin{equation*}
W^{p}_{h}(S_{k,h}([\Gamma']))=\frac{4\pi k^{\frac{p}{2}}}{(1-k)}\mathrm{g}(\Gamma')\ .
\end{equation*}

We conclude that, for every $\eps,w>0$,
\begin{equation*}
\#\{[\Gamma']\in\QF(1+\eps):W^{p}_{h}(S_{k,h}([\Gamma']))\leq w\}\geq \#\{[\Gamma']\in\sigma(\Gamma):W^{p}_{h}(S_{k,h}([\Gamma']))\leq w\}\ .
\end{equation*}
Thus, by Lemma \ref{l.Muller_Puchta},
\begin{equation*}
\Ent^{p}_{k}(X,h)\geq\frac{(1-k)k^{-\frac{p}{2}}}{2\pi}=\Ent^{p}_{k}(X,h_0)\ ,
\end{equation*}
and the result follows.
\end{proof}

\subsubsection{The modified energy entropy}\label{sss.Modified_entropy}
In light of Theorem \ref{th.totally_umb_implies_big_entropy}, it is reasonable to exclude the counting of quasi-Fuchsian surfaces that accumulate to a Fuchsian one. This motivates modifying our energy entropy for $k$-surfaces, in line with the modified entropy of \cite{MN_currents}. This modified entropy will focus on counting only those quasi-Fuchsian surfaces that accumulate to the unique ergodic fully supported laminar measure in the sphere bundle $S_hX$.

 Let $\hat m$ denote the Haar measure of the set $\cC^+$ of round circles of $\C\PP^1$. Recall from Theorem \ref{thm.Ratner} that this is the unique fully supported, ergodic conformal current in $\cC^+$. Note now that every conformal current is determined by its restriction to some compact fundamental domain, and it follows that, for all $C\geq 1$, the space of conformal currents over $\QC^+(C)$ is metrizable. Let $\rho$ be such a metric. Given $\eta>0$, we define
\begin{equation*}
\QF_{\hat m}(\eta):=\left\{[\Gamma]\in\QF:\rho(\hat m(\bord\Gamma),\hat m)<\eta\right\}\ .
\end{equation*}
Following \cite{MN_currents} we define the \emph{modified $p$-energy entropy} by
\begin{equation*}
\Ent^p_{k,\hat m}(X,h):=\lim_{\eta\to 0}\liminf_{w\to\infty}\frac{1}{w\log w}\log\#\left\{[\Gamma]\in\QF_{\hat m}(\eta):W^{p}_{h}(S_{k,h}([\Gamma]))\leq w\right\}\ .
\end{equation*}
This entropy counts only quasi-Fuchsian $k$-surfaces that accumulate to the fully-supported, ergodic laminar measure. In particular, when $\Pi$ has no Fuchsian subgroups (see, for example, \cite[\S 5.3]{MR03}),
\begin{equation*}
\Ent^{p}_{k}(X,h)=\Ent^{p}_{k,\hat m}(X,h)\ .
\end{equation*}

\begin{theorem}
Assume that $-1\leq\sect_h\leq-a$, $p\geq 0$ and $0<k<a$. Then
\begin{equation*}
0<\Ent^{p}_{k,\hat m}(X,h)\leq\Ent^{p}_{k}(X,h)\ .
\end{equation*}
\end{theorem}

\begin{proof} Let $(\Gamma_n)_{n\in\mathbb{N}}$ be a sequence of quasi-Fuchsian subgroups of $\Pi$ as in Theorem \ref{th.lowe_neves}. By compactness of $S_hX$, the mean curvatures of the leaves of the foliation $\cF_{k,h}$ are uniformly bounded. An argument by contradiction similar to that used in Lemma \ref{l.almost_tot_umbilical} then shows that there exists a uniform upper bound $H>0$ for the mean curvatures of all Lowe--Neves $k$-surfaces $S_n=S_{k,h}([\Gamma_n])$. Arguing as in Lemma \ref{lem_lower_bounds_energy} we then obtain, for every $n\in\N$,
\begin{equation}\label{eq.upperbound_energy}
W^{p}_{h}(S_n)\leq \frac{4\pi H^p}{(a-k)}(\mathrm{g}(\Gamma_n)-1)\ .
\end{equation}
Now fix $\eta>0$ and $n_0$ such that $[\Gamma_{n_0}]\in\QF_{\hat m}(\eta)$. Reasoning as in the proof of Theorem \ref{th.totally_umb_implies_big_entropy}, we show that, for every $[\Gamma']\in\sigma(\Gamma_{n_0})$,
\begin{equation*}
[\Gamma']\in\QF_{\hat m}(\eta)\ .
\end{equation*}
Furthermore, since we have the same bounds on the mean and extrinsic curvatures, \eqref{eq.upperbound_energy} also holds for $S'=S_{k,h}([\Gamma'])$. This implies that, for sufficiently large $g>0$,
\begin{equation*}
\#\{[\Gamma']\in\sigma(\Gamma_{n_0})\ \big|\ \mathrm{g}(\Gamma')\leq g\}\leq\#\{[\Gamma]\in\QF_{\hat m}(\eta)\ \big|\ W^{p}_{h}(S_{k,h}([\Gamma]))\leq \frac{4\pi H^p}{(a-k)}(g-1)\}\ .
\end{equation*}
Lemma \ref{l.Muller_Puchta} now yields a positive lower bound for the quantity
\begin{equation*}
\liminf_{w\to\infty}\frac{1}{w\log(w)}\log\#\{[\Gamma]\in\QF_\mu(\eta)\ \big|\ W^{p}_{h}(S_{k,h}([\Gamma]))\leq w\}
\end{equation*}
which only depends on $H$ and $a$, but not on $\eta$. This yields $0<\Ent^p_{k,\hat m}(X,h)$, as desired.
\end{proof}

\subsection{Rigidity of the modified entropy}

\subsubsection{Proof of Theorem \ref{th.rigidity_entropy}} We assume from now on that $-1\leq\sect_h\leq -a <0$. Take $p\geq 0$ and $0<k<a$. We first need to prove the following lemma.

\begin{lemma}\label{lem_converging_sequence}
Suppose
\begin{equation*}
\Ent^{p}_{k,\hat m}(X,h)=\Ent^{p}_{k,\hat m}(X,h_0)\ .
\end{equation*}
Then there exists a sequence $(\eta_n)_{n\in\mathbb{N}}$ converging to zero, and a sequence $(\Gamma_n)_{n\in\mathbb{N}}$ of quasi-Fuchsian subgroups of $\Pi$ such that,
\begin{itemize}
\item for all $n$, $[\Gamma_n]\in\QF_{\hat m}(\eta_n)$; and
\item
\begin{equation*}
\lim_{n\to\infty}\frac{W^{p}_{h}(S_{k,h}([\Gamma_n]))}{W^{p}_{h_0}(S_{k,h_0}([\Gamma_n]))}=1\ .
\end{equation*}
\end{itemize}
\end{lemma}

\begin{proof}
Suppose the contrary. Then there exists $\delta>0$ such that for all sufficiently small $\eta>0$, and for all $[\Gamma]\in\QF_{\hat m}(\eta)$,
\begin{equation*}
W^{p}_{h}(S_{k,h}([\Gamma]))\geq (1+\eta) W^{p}_{h}(S_{k,h_0}([\Gamma]))\ .
\end{equation*}
This immediately implies
\begin{equation*}
\Ent^p_{k,\hat m}(X,h)\leq\frac{1}{(1+\eta)}\Ent^{p}_{k,\hat m}(X,h_0)\ ,
\end{equation*}
which is absurd, and the result follows.
\end{proof}

We now prove Theorem \ref{th.rigidity_entropy}.

\begin{proof}[Proof of Theorem \ref{th.rigidity_entropy}] Assume that $\Ent^{p}_{k,\hat m}(X,h)=\Ent^{p}_{k,\hat m}(X,h_0)$. Then Lemma \ref{lem_converging_sequence} provides a sequence $(\eta_n)_{n\in\mathbb{N}}$ converging to zero as well as a sequence $([\Gamma_n])_{n\in\mathbb{N}}\in\QF_{\hat m}(\eta_n)$ such that
\begin{equation*}
\lim_{n\to\infty}\frac{W^{p}_{h}(S_{k,h}([\Gamma_n]))}{W^{p}_{h_0}(S_{k,h_0}([\Gamma_n]))}=1\ .
\end{equation*}

By definition of the sequence $([\Gamma_n])_{n\in\mathbb{N}}$, the sequence $(\mu_{k,h}(\bord \Gamma_n))_{n\in\mathbb{N}}$ of laminar measures converges to $\mu$, the unique ergodic laminar measure that has full support in $S_hX$.

By reproducing the argument given in \S \ref{ss.Rigidity} we get that $\sect_h=-1$ almost everywhere for $\mu$. So we must have $\sect_h=-1$ everywhere and $h$ is isometric to $h_0$ by Mostow's rigidity theorem.
\end{proof}

\subsubsection{Case of equality between modified energy and area entropies} \label{sss.equality_modified_entropies}

As in \cite{ALS,MN_currents} we can define the \emph{modified area entropy} as
\begin{equation*}
\Ent^{\text{Area}}_{k,\hat m}(X,h)=\lim_{\eta\to 0}\liminf_{A\to\infty}\frac{1}{A\log A}\log\#\left\{[\Gamma]\in\QF_{\hat m}(\eta)\ \bigg|\ \Area(S_{k,h}([\Gamma]))\leq A\right\}\ .
\end{equation*}

We give the proof of our last theorem.

\begin{proof}[Proof of Theorem \ref{equality_energy_area}.]
We will first assume $\Ent^{\text{Area}}_{\hat m}(M,h)=\Ent^{1}_{k,\hat m}(M,h)$. Then an argument similar to Lemma \ref{lem_converging_sequence} gives the existence of a sequence $(\eta_n)_{n\in\mathbb{N}}$ converging to zero as well as a sequence $([\Gamma_n])_{n\in\mathbb{N}}\in\QF_{\hat m}(\eta_n)$ such that
\begin{equation*}
\lim_{n\to\infty}\frac{k^{-\frac{1}{2}}W^{1}_{h}(S_{k,h}([\Gamma_n]))}{\text{Area}_h(S_{k,h}([\Gamma_n]))}=1\ .
\end{equation*}
Here again, by definition of the sequence $([\Gamma_n])_{n\in\mathbb{N}}$, the sequence $(\mu_{k,h}(\bord \Gamma_n))_{n\in\mathbb{N}}$ of laminar measures converges to the ergodic laminar measure $\mu$, which has full support in $S_hX$.

Hence by reproducing the argument in \S \ref{ss.proof_of_thm_B} we obtain that the $k$-surface foliation of $S_hX$ is by totally umbilical leaves. So by applying the axiom of $2$-spheres we conclude that the sectional curvature of $h$ is constant.

The other implications follow from the same argument.
\end{proof}

\begin{remark} \label{remarkattheend}
We comment that the results of this section also hold with the modified entropy functionals replaced by entropy functionals that count k-surfaces whose limit sets in the hyperbolic metric become more and more round, but without a condition on equidistribution, provided that $X$  contains no closed totally geodesic surfaces.  The condition that there be no closed totally geodesic surfaces causes the equidistribution to hold automatically. (See \cite{MN_currents} for a closely analogous situation for entropy functionals counting minimal surfaces rather than $k$-surfaces.)

If the following question had a positive answer, it would be possible to drop the condition that $X$ contain no closed totally geodesic surface in its hyperbolic metric.  If it had a negative answer, it would be possible to give counterexamples to statements involving the un-modified k-surface entropy by taking covers of a totally umbilic surface as in the statement below.
\end{remark}

\begin{question}
Suppose that $X$ has sectional curvature greater than or equal to $-1$, and that $X$ contains a closed totally umbilic surface $S$ with mean curvature $k^{1/2}$ and with constant curvature $-1 +k$ in its intrinsic metric (i.e., the sectional curvature of $X$ through every tangent plane to $S$ is $-1$.)  Then must $X$ be hyperbolic?
\end{question}

\begin{flushleft}
{\scshape S\'ebastien Alvarez}\\
	CMAT, Facultad de Ciencias, Universidad de la Rep\'ublica, \& IRL-IFUMI (CNRS)\\
	Igua 4225 esq. Mataojo. Montevideo, Uruguay.\\
	email: \texttt{salvarez@cmat.edu.uy}

	\vspace{0.2cm}

{\scshape Ben Lowe}\\
  	Department of Mathematics, University of Chicago,
    Chicago IL 60637, USA.\\
	email: \texttt{loweb24@uchicago.edu}
	\vspace{0.2cm}
	
{\scshape Graham Andrew Smith}\\
	Departamento de Matem\'atica,\\
    Pontifícia Universidade Católica do Rio de Janeiro,\\
	Rua Marqu\^es de S\~ao Vicente 225, G\'avea,\\
    Rio de Janeiro 225453-900, Brazil.\\
	email: \texttt{grahamandrewsmith@gmail.com}	
	\vspace{0.2cm}

\end{flushleft}

\end{document}